\documentclass[12pt,oneside,a4paper,reqno]{amsart}

\usepackage[utf8]{inputenc}
\usepackage[margin=2.5cm]{geometry}
\usepackage[foot]{amsaddr}
\usepackage{parskip}
\usepackage{amsmath}
\usepackage{amssymb}
\usepackage{amsthm}
\usepackage{mathtools}
\usepackage{bbm}
\usepackage{aligned-overset}

\usepackage[sortcites,url=false,eprint=false,giveninits=true]{biblatex}
\usepackage[colorlinks,allcolors=blue]{hyperref}
\numberwithin{equation}{section}
\newtheorem{theorem}{Theorem}[section]
\newtheorem{lemma}[theorem]{Lemma}
\newtheorem{proposition}[theorem]{Proposition}
\newtheorem{definition}[theorem]{Definition}
\newtheorem{corollary}[theorem]{Corollary}
\newtheorem{remark}[theorem]{Remark}

\DeclareMathOperator{\Real}{Re \,}
\renewcommand{\d}{{\, \mathrm d}}
\newcommand{\eps}{\varepsilon}
\newcommand{\E}{\mathbb{E}}
\newcommand{\PP}{\mathbb{P}}
\newcommand{\C}{\mathbb{C}}
\newcommand{\R}{\mathbb{R}}
\newcommand{\N}{{\mathbb{N}}}

\title{Solitary waves in a stochastic parametrically forced nonlinear Schr\"odinger equation}

\date{\today}

\author{Manuel V.~Gnann\textsuperscript{1}}
\address{\textsuperscript{1}Delft Institute of Applied Mathematics, Faculty of Electrical Engineering, Mathematics and Computer Science, Delft University of Technology, Mekelweg 4, 2628 CD Delft, Netherlands}
\email{M.V.Gnann@tudelft.nl}

\author{Rik W.S.~Westdorp\textsuperscript{2}}
\address{\textsuperscript{2}Mathematical Institute, Leiden University, P.O. Box 9512, 2300 RA Leiden, The Netherlands}
\email{R.W.S.Westdorp@math.leidenuniv.nl}

\author{Joris van~Winden\textsuperscript{1}}
\email{J.vanWinden@tudelft.nl}



\thanks{This work supersedes \url{https://arxiv.org/abs/2208.01945}. Early versions of the results contained in this paper can be found in the MSc theses of the second and third author, both prepared under the supervision of the first author at Delft University of Technology. The second author acknowledges support from the Dutch Research Council (NWO) (grant 613.009.137). The third author is supported by a DIAM fast-track scholarship. The authors thank Mark Veraar for discussions and valuable suggestions on the manuscript.}

\begin{document}

\begin{abstract}
We study a parametrically forced nonlinear Schr\"odinger (PFNLS) equation, driven by multiplicative translation-invariant noise.
We show that a solitary wave in the stochastic equation is orbitally stable on a timescale which is exponential in the inverse square of the noise strength.
We give explicit expressions for the phase shift and fluctuations around the shifted wave which are accurate to second order in the noise strength.
This is done by developing a new perspective on the phase-lag method introduced by Kr\"uger and Stannat.
Additionally, we show well-posedness of the equation in the fractional Bessel space $H^{s}$ for any $s \in [0,\infty)$, demonstrating persistence of regularity.
\end{abstract}

\maketitle

\section{Introduction}
\subsection{The parametrically forced nonlinear Schr\"odinger equation}
Optic fibers that act as waveguides for electromagnetic signals form the basis for systems of fiber-optic communications, enabling long-distance communication at high bandwidth~\cite{agrawal2012}. 
The behavior of a pulse propagating through an optic fiber is governed by the nonlinear Schr\"odinger (NLS) equation~\cite{agrawal2000}, which is an archetypal example of a nonlinear dispersive equation that is known to support solitary waves. 
The NLS equation has many applications in physics, for instance in the description of Bose--Einstein condensates~\cite{bronski}, deep-water waves~\cite{vitanov}, and plasma oscillations~\cite{shukla}. 
In these applications, the NLS equation describes the complex amplitude of a wave packet propagating through a nonlinear medium. 
We refer to~\cite{sulem} for a detailed treatment of the physical background. 

In optic fibers, the nonlinear behavior arises due to a response of the refractive index of the fiber to an applied electric field known as the Kerr effect, leading to a cubic nonlinear term in the equation.
Effective signal transmission in optic communication systems may be obstructed by the presence of linear loss in the fiber, weakening the signal as it propagates. 
\citeauthor{kutz} proposed a method of compensating loss using periodic phase-sensitive amplification~\cite{kutz}, which has since become a popular approach for increasing feasible transmission lengths. 
The approach is modelled by the parametrically forced nonlinear Schr\"odinger (PFNLS) equation:
\begin{equation}
    \label{eq:pfnls}
    \d u = (i\Delta u -i \nu u -\epsilon (\gamma u -\mu \overline{u}))\d t+i\kappa \lvert u\rvert^2 u \d t \quad \text{for } (x,t)\in \R \times \R^+.
\end{equation}
Here, the complex-valued function $u(x,t)$ denotes the envelope of the electric field in an optic fiber, $t$ is the distance along the fiber, and $x$ denotes time in a translating frame that moves with the group velocity of light. 
The constants $\gamma>0$ and $\mu>0$ model linear loss in the fiber and phase-sensitive amplification, respectively. 
The constant $\nu \in \R$ models a phase advance of the signal carrier, and the constant $\kappa>0$ denotes the strength of the Kerr effect in the fiber. 
In this model, the local effect of the periodically spaced phase-sensitive amplifiers is averaged over the spacing length of the amplifiers. This description assumes that the amplifiers are closely spaced, which is valid for long propagation lengths~\cite{kath}. 
In particular, the model applies well to a re-circulating loop used for long-term storage of pulses in optical networks.

In case that $\mu>\gamma$, i.e. enough amplification is supplied, equation \eqref{eq:pfnls} admits solitary standing wave solutions $u^*$ of the form
\begin{align}
    \label{eq:soliton}
    u^*(x)= \sqrt{\frac{2 (\nu+\epsilon \mu \sin(2\theta))}{\kappa}}\text{sech}(\sqrt{\nu+\epsilon \mu \sin(2\theta)}x)e^{i\theta}, 
\end{align}
where $\theta \in [0, 2\pi)$ is a solution to $\cos(2\theta)=\gamma/\mu$.
This can be seen from~\cite[equation (1.8)]{kapitula_stability_1998} after scaling in $\kappa$ by setting $\phi = \tfrac{1}{2}\sqrt{\kappa} u$.
As equation \eqref{eq:pfnls} is translation invariant, shifting the solitary waves by an arbitrary constant $a\in \R$ produces a family of solutions.
The solitary waves for which  $\sin(2\theta)>0$ were shown to be orbitally exponentially stable by Kapitula and Sandstede~\cite{kapitula_stability_1998}: small perturbations of the solitary wave converge at an exponential rate to a suitable translate of the solitary wave. 
Solitary waves for which $\sin(2\theta)<0$ are known to be unstable~\cite{kutzkath}.

We briefly note that in the physical application of optic fiber loops, the term standing wave is misleading, as the equation describes the electric field in a moving frame. 
The standing waves \eqref{eq:soliton} represent traveling pulses, and their stability is crucial for attaining long transmission lengths of signals and for the feasibility of long-time storage.

The stability analysis in~\cite{kapitula_stability_1998} relies on computing the spectrum of the (real-)linear operator
\begin{align*}
        \mathcal{L}v = i\Delta v -i \nu v -\epsilon (\gamma v -\mu \overline{v})+i \kappa (2\lvert u^*\rvert^2 v+(u^*)^2 \overline{v})
\end{align*}
on $L^2(\R;\C)$ associated with the linearization of \eqref{eq:pfnls} around the solitary wave. 
It is known that the spectrum of the linearization is located at an $\mathcal{O}(\epsilon)$ distance to the left of the imaginary axis, except for a simple eigenvalue at zero~\cite{alexander,kapitula_stability_1998}. 
This eigenvalue arises due to the translation invariance of \eqref{eq:pfnls}. 
For $\epsilon=\nu=0$, the operator $\mathcal{L}$ corresponds to the linearization around the primary soliton in the NLS equation, and has continuous spectrum on the imaginary axis. 
The primary NLS soliton is also orbitally stable, but no exponential decay of perturbations can be expected~\cite{weinstein, pelinovsky}. 
As such, parametric forcing entails stronger linear stability.

\subsection{A stochastic equation}
In~\cite{kath}, \citeauthor{kath} discuss two mechanisms that further inhibit signal transmission by introducing noise in the system, thereby transforming the description of pulse propagation into a stochastic partial differential equation. 
In this paper, we study the evolution of the solitary wave $u^*$ \eqref{eq:soliton} in the stochastic parametrically forced nonlinear Schr\"odinger (SPFNLS) equation:
\begin{equation}
    \label{eq:spfnls}
    \d u = (i\Delta u -i \nu u -\epsilon (\gamma u -\mu \overline{u}))\d t+i\kappa|u|^2 u \d t -i u \circ (\phi * \d W) \quad \text{for } (x,t)\in \R \times \R^+.
\end{equation}
The symbol $W$ denotes a cylindrical Wiener process in the Hilbert space $L^2(\R,\R)$, meaning that $\d W$ is a space-time white noise, and $\circ$ denotes the Stratonovich product.
Here, $\phi$ is a real-valued function, which serves to regularize the noise.
Thus, $u$ is multiplied by noise which is white in time, and formally satisfies the covariance relation $\E[\d W(t,x) \d W(t,y)] = \tilde{\phi} * \phi(y-x)$ in space ($\tilde{\phi}$ denotes the reflection of $\phi$ around the origin).
Because the covariance only depends on $y-x$, equation \eqref{eq:spfnls} preserves the physically relevant symmetry of translation invariance (in a statistical sense).
This is highly relevant to our study of the motion of solitary waves.

The multiplicative noise term that we consider in \eqref{eq:spfnls} models phase noise induced by the coupling of light with the thermally excited acoustical modes of the fiber known as guided acoustic-wave Brillouin scattering (GAWBS)~\cite{kath}. 
We use the Stratonovich product, as it is more realistic for physical applications. Indeed, in the absence of parametric forcing, it allows for conservation of the $L^2(\R)$-norm~\cite[Proposition 4.1]{de_bouard_stochastic_1999}. 
Because our variable $x$ corresponds to physical time, our noise is correlated in time, which is a natural assumption in the context of GAWBS phase noise. 
The other noise effect proposed in~\cite{kath} is due to quantum effects and results in an additive noise term. We focus in the present paper only on the multiplicative GAWBS phase noise.

\subsection{Well-posedness}
Our first result concerns well-posedness of the stochastic equation \eqref{eq:spfnls}.
We show that for any $s \geq 0$, $\phi$ in the fractional Bessel space $H^s(\R;\R)$ and $u(0) \in H^s_x$, equation \eqref{eq:spfnls} has a unique mild solution $u$ taking values in the space ${C([0,T];H^s_x) \cap L^r(0,T;L^p_x)}$ for every $T > 0$ and certain pairs $(p,r)$ (see Theorem~\ref{thm:wellposed} and Definition~\ref{def:admissible}).

The `standard' SNLS equation with linear multiplicative noise (corresponding to the case $\epsilon=\nu=\gamma=\mu=0$) was first shown to be well-posed in the spaces $L^2_x$ (corresponding to $s = 0$)~\cite{de_bouard_stochastic_1999} and $H^1_x$ (corresponding to $s = 1$)~\cite{de_bouard_stochastic_2003}.
A proof of the $L^2_x$ well-posedness using stochastic Strichartz estimates is given in~\cite{hornung_nonlinear_2018}.
Since the PFNLS equation differs from the NLS equation by linear terms, our proof of well-posedness is very similar.
The main novelties are well-posedness in $H^s(\R;\R)$ for $s \in [0,\infty)\setminus\{0,1\}$ and the use of translation-invariant noise.
The translation-invariant noise, aside from being motivated by physical symmetries, is relevant to our subsequent study of the solitary waves and is not directly covered by previous results.
The well-posedness in $H^s_x$ shows that, like its deterministic counterpart, the SPFNLS (and by extension, the one-dimensional cubic SNLS) equation has \emph{persistence of regularity}, meaning that regularity of the solution is the same as the minimum of that of the noise and the initial data.
Previous results on stochastic versions of these equations have mainly been concerned with the cases $s = 0$ and $s = 1$.
\subsection{Orbital stability}
With the well-posedness of \eqref{eq:spfnls} firmly established, we turn to the stability of the solitary wave $u^*$ with $\sin(2\theta) > 0$ (see the discussion following \eqref{eq:soliton}) in the stochastic equation.
We establish that the solitary wave is orbitally stable under the multiplicative stochastic forcing in \eqref{eq:spfnls} on a timescale $T \sim \exp(\sigma^{-2})$, where $\sigma$ denotes the strength of the noise.
We describe the solution to \eqref{eq:spfnls} with initial condition close to $u^*$ using the decomposition
\begin{equation*}
    u(x,t)=u^*(x+a(t))+v(x,t),
\end{equation*}
where $a$ is a real-valued stochastic process that tracks the wave position, and $v$ an infinite-dimensional perturbation which is small when measured in the $L^2_x$-norm.
In the parabolic setting, such problems are well-studied (see e.g.~\cite{kruger_front_2014,inglis_general_2016,maclaurin_phase_2023,hamster_stability_2019}).
Rigorous results in a dispersive setting are more scarce~\cite{fukuizumi,westdorp_long-timescale_2024} and, as far as we are aware, stability on exponential timescales has not been shown before.

We give explicit expressions for $a(t)$ and $v(t)$ which are accurate to second order in $\sigma$.
Second-order results in this setting are scarce, and mostly consist of formal computations~\cite{lang_multiscale_2016}.
By developing a new perspective on an established phase-tracking method (see Section~\ref{sec:tracking}) we rigorously and efficiently prove accuracy of the second-order expressions for the first time.

To first order, the phase process $a(t)$ behaves like a Brownian motion with variance proportional to $t \sigma^2$, and the perturbation $v(t,x)$ behaves like an infinite-dimensional Ornstein-Uhlenbeck process.
More precisely, $v(t,x)$ is mean-reverting and satisfies an estimate of the form
\begin{equation}
    \mathbb{E}\bigl[\lVert v(t)\rVert_{L^2_x}^2\bigr]^{1/2} \leq C\sigma(e^{-at}\|v(0)\|_{L^2} +\min\{t^{\frac{1}{2}},1\})+\mathcal{O}(\sigma^2)
\end{equation}
(see Theorem~\ref{thm:relaxation}).
Using such bounds to control the development of a perturbation over short time-scales combined with a resetting procedure, we show that there exists a stochastic process $a(t)$ and constants $C,k,\eps' > 0$ such that
\begin{equation*}
    \mathbb{P}\Bigl[\,\sup_{t \in [0,T]} \lVert u(\cdot,t)-u^*(\cdot+a(t))\rVert_{L^2_x} \geq \eps\Bigr]\leq C Te^{-k\sigma^{-2}\eps^2}
\end{equation*}
for all $T>0$ and $0 < \sigma \leq \eps \leq \eps'$ (Proposition~\ref{prop:longterm} and Corollary~\ref{cor:longterm}).
This shows stability on a timescale $T \sim e^{k \sigma^{-2}\eps^2}$.
By a scaling argument, this is (up to better constants) the longest time for which the solitary wave can be expected to be stable, and matches the best results obtained in different settings, such as~\cite{hamster_stability_2020-1,maclaurin_phase_2023}.

\subsection{Phase tracking}
\label{sec:tracking}
When showing stochastic orbital stability, there are several different ways of defining and tracking the phase process $a(t)$ (see e.g.~\cite{kruger_front_2014,inglis_general_2016,hamster_stability_2019}).
Our method is closely related to the one developed by \citeauthor{kruger_front_2014}~\cite{kruger_front_2014,kruger_multiscale-analysis_2017}, which has also been applied by \citeauthor{eichinger_multiscale_2022} to the FitzHugh--Nagumo equation~\cite{eichinger_multiscale_2022}.
Briefly, this method consists of defining an approximation process $a_m(t)$ using the random ODE
\begin{equation*}
    \frac{\d a_m(t)}{\d t} = - m \frac{\partial \lVert u(t,x) - u^*(t,x + a_m(t)) \rVert_{L^2_x}}{\partial a_m},
\end{equation*}
and computing an SDE for $\frac{\d a_m(t)}{\d t}$.
By approximating the SDE to first order in $\sigma$ and taking $m \to \infty$, orbital stability can be shown on timescales of the order $T \sim \sigma^{-2}$.

Our method obtains a similar phase process via a completely different route, which we briefly summarize.
Before introducing our phase process, we first prove an asymptotic expansion of the form
\begin{equation}
    \label{eq:introexpansion}
    u(t,x) = u^*(x) + \sigma v_1(t,x) + \sigma^2 v_2(t,x) + \mathcal{O}(\sigma^3)
\end{equation}
(Theorem~\ref{thm:asym}).
This results in explicit representations of $v_1$ and $v_2$, as well as exact estimates relating to the validity of the expansion.
Since the PFNLS equation is not parabolic, we rely on dispersive estimates to control the nonlinear terms.
We also require Gaussian tail estimates on the remainder terms, for which we use a result by Seidler~\cite{seidler_exponential_2010} to estimate $L^p_{\Omega}$-norms of stochastic integrals with a constant which is $\mathcal{O}(\sqrt{p})$.

The next step is to introduce the following decomposition of $v_1$ and $v_2$:
\begin{subequations}
    \label{eq:introdecomp}
    \begin{align}
        v_1(t,x) &= w_1(t,x) + a_1(t) u^*_x(x), \\
        v_2(t,x) &= w_2(t,x) + a_2(t) u^*_x(x) + \tfrac{1}{2}a_1(t)^2 u^*_{xx}(x),
    \end{align}
\end{subequations}
where $w_1$ and $w_2$ should be regarded as being determined by \eqref{eq:introdecomp} for a given choice of $a_1$ and $a_2$.
We show that there are unique choices of $a_1$ and $a_2$ such that the linear parts of $w_1$ and $w_2$ are mean reverting,
and we take these to be our definition of the first- and second-order components of the phase (see Section~\ref{sec:orbital} and Theorem~\ref{thm:relaxation}).
This allows us to use deterministic linear stability results to show that $w_1$ does not show any growth in time, and $w_2$ grows at a slower rate than $v_2$.
Directly combining the asymptotic expansion \eqref{eq:introexpansion} with the decomposition \eqref{eq:introdecomp} using a Taylor expansion finally results in
\begin{align*}
    u(t,x) &= u^*(x + \sigma a_1(t) +\sigma^2 a_2(t)) + \sigma w_1(t,x)+\sigma^2 w_2(t,x) + \mathcal{O}(\sigma^3),
\end{align*}
which, combined with smallness of $w_1$ and $w_2$, shows orbital stability on a timescale for which the asymptotic expansion \eqref{eq:introexpansion} is valid.

Asserting stability on longer timescales requires additional effort.
The main issue is that \eqref{eq:introexpansion} is a linearization around $u^*$, but after time $t$ the solution is close to the translated wave $u^*(x+a(t))$.
Thus, when $a(t)$ gets large enough (which happens on a timescale $T \sim \sigma^{-2}$), the linearization becomes completely inaccurate.
We remedy this by resetting the linearization after a fixed time $T$, by linearizing around the shifted wave $u^*(x+a(T))$ instead.
This makes it possible to combine the short-term estimates on each time interval $[NT,(N+1)T]$ to obtain long-term stability (Corollary~\ref{cor:longterm}).
The cost of this procedure is that we incur a discontinuity in the phase process each time we reset, and our explicit representation is only valid in between resetting.
We are not aware of any methods to obtain \emph{explicit} descriptions of the phase which are accurate on long timescales.
Surprisingly, the resetting procedure suggests that it is possible to show stability on long timescales without accurately tracking the phase on short timescales.
This is something we aim to investigate in future work.

\subsection{Outline}
In Section~\ref{sec:prelim} we specify our notation and introduce the preliminaries necessary to state and prove the main results (Theorems~\ref{thm:wellposed},~\ref{thm:asym},~\ref{thm:relaxation}, and Propostion~\ref{prop:longterm}), which are contained in Section~\ref{sec:mainresult}.
The proof of well-posedness of \eqref{eq:spfnls} is given in Section~\ref{sec:proofwellposed}, followed by the proof of the stability results in Section~\ref{sec:proofstability}.
Appendices~\ref{app:hs} and~\ref{app:stochstrich} contain some auxiliary results needed for the proofs.

\section{Preliminaries}
\label{sec:prelim}
We now give the preliminaries required to state and prove the main results, as well as some notational shorthands.
We give a rigorous meaning to \eqref{eq:spfnls}, and formulate the Strichartz estimates which are used to show well-posedness.
Afterwards we state the deterministic stability of the solitary wave, along with additional Strichartz estimates related to the linearization around the solitary wave, which are needed for our stochastic stability results.
\subsection{Notation and conventions}
We denote the norm of general normed spaces $X$ by $\lVert \cdot \rVert_X$, and the inner product of general inner product spaces $H$ by $\langle \cdot, \cdot \rangle_{H}$.
In the case where $H$ is complex, we take the inner product to be conjugate-linear in the second variable.
The space of bounded linear operators from a Banach space $X$ to a Banach space $Y$ is denoted by $\mathcal{L}(X;Y)$, and the space of Hilbert--Schmidt operators between separable Hilbert spaces $H$ and $\tilde{H}$ as $\mathcal{L}_2(H;\tilde{H})$.
If a mapping $F$ between two Banach spaces $X$ and $Y$ is $n$ times Fr\'echet differentiable at a point $x_0\in X$, then we denote its Fr\'echet derivative at $x_0$ by $(h_1, \ldots, h_n) \mapsto \d F(x_0)[h_1,\dots,h_n]$.

If $X$ is a Banach space, we will write $C([0,T];X)$ for the space of continuous $X$-valued functions.
For $p \in [1,\infty]$, we write $L^p(S;X)$ for the usual Bochner spaces defined on a measure space $(S,\mathcal{F},\mu)$ (which coincide with the Lebesgue spaces if $X = \C$ or $X = \R$).
If $p = 2$, and $H$ is a Hilbert space, then $L^2(S;H)$ is a Hilbert space with the inner product given by
\begin{align*}
    \langle f, g \rangle_{L^2(S;X)} = \int_{S} \langle f, g \rangle_H \d \mu.
\end{align*}
For $z \in \C$, we write $\overline{z}$ for its complex conjugate.
For $p\in[1,\infty]$, we write $p^\prime$ for its H\"older conjugate, which is the unique $p' \in [1,\infty]$ such that $\frac{1}{p}+\frac{1}{p'}=1$.
Throughout the paper, all random variables will be defined on a complete probability space $(\Omega,\mathcal{F}, \mathbb{P})$ equipped with a complete and right-continuous filtration $\mathbb{F} = (\mathcal{F}_t)_{t \in [0,\infty)}$.
We will make use of the following abbreviations:
\begin{align*}
    L^p_x &\coloneqq L^p(\R;\C), \\
    L^p_{\Omega}(X) &\coloneqq L^p(\Omega;X), \\
    L^p(T,T';X) &\coloneqq L^p([T,T'];X),
\end{align*}
where $\R$ and $[T,T']$ are equipped with the usual Lebesgue measure.

The weak derivative of a weakly differentiable function $f\in L^p_x$ is denoted by $\partial_x f$ and we write $\Delta=\partial_x^2$ for the Laplacian on the real line.
We write $u^*_x$ and $u^*_{xx}$ for the first and second spatial derivatives of $u^*$.
For $s \in [0,\infty)$ and $p \in (1,\infty)$, the Bessel space $H^{s,p}_x$ consists of the functions $f \in L^p_x$ for which the quantity
\begin{align*}
        \lVert f \rVert_{H^{s,p}_x} = \lVert (1-\Delta)^{\frac{s}{2}}f \rVert_{L^p_x}
\end{align*}
is finite.
Here, the fractional power $(1-\Delta)^{\frac{s}{2}}$ is defined using the Fourier multiplier with symbol $\xi \mapsto (1+\lvert \xi \rvert^2)^{\frac{s}{2}}$.
The space $H^{s,p}_x$ is a Banach space and we have continuous embeddings $H^{s_1,p}_x \hookrightarrow H^{s_2,p}_x$ if $s_1 \geq s_2$.
When $k$ is a nonnegative integer, the Bessel space $H^{k,p}_x$ is isomorphic to the classical Sobolev space $W^{k,p}_x$, which consists of the function in $L^p_x$ for which all partial derivatives of order $k$ or less are also in $L^p_x$.
Proofs of these statements rely on the theory of singular integrals, and can for example be found in~\cite[Chapter 3]{stein_singular_1970}.
We also note that $H^{s,2}_x$ is a Hilbert space with inner product $\langle f, g \rangle_{H^{s,2}_x} = \langle (1-\Delta)^{\frac{s}{2}}f,(1- \Delta)^{\frac{s}{2}}g \rangle_{L^2_x}$.
In this case we will write $H^s_x \coloneqq H^{s,2}_x$.

Lastly, we denote by $\{S(t)\}_{t\in \R}$ the $C_0$-group on $L^2_x$ generated by ${i \Delta: L^2_x \supset H^2_x \to L^2_x}$, which acts at $t\in \R$ as the Fourier multiplier with symbol $\xi \mapsto e^{-4\pi^2 i|\xi|^2 t}$.
Using Plancherel's theorem, it can be seen that $S(t)$ is unitary on $L^2_x$.
Since the Fourier multiplier of $S(t)$ commutes with that of $(1-\Delta)^{\frac{s}{2}}$, it is immediate that $S(t)$ is also a unitary group on $H^{s}_x$ for any $s$.
\subsection{Stochastic set-up}
We let $W(t)$ be an $L^2(\R;\R)$-cylindrical Wiener process on $\Omega$, which is adapted to $\mathbb{F}$.
Then $W(t)$ has an interpretation as the time integral from $0$ to $t$ over a space-time white noise.
To regularize the noise, fix some $\phi \in L^2(\R;\R)$ and define $\Phi \colon L^2(\R;\R) \rightarrow L^{\infty}_x$ and $\beta \in \R$ as
\begin{subequations}
\label{eq:defphi}
\begin{align}
\label{eq:defPhi}
    \Phi f &\coloneqq \phi * f, \\
\label{eq:defbeta}
    \beta &\coloneqq \lVert \phi \rVert_{L^2_x}.
\end{align}
\end{subequations}
We now convert \eqref{eq:spfnls} into an equivalent formulation in It\^o form.
Formally applying an It\^o--Stratonovich correction to \eqref{eq:spfnls} results in
\begin{equation}
    \label{eq:spfnlsitoformal}
    \d u = [i \Delta u - i \nu u - \epsilon (\gamma u - \mu \overline{u}) + i \kappa \lvert u \rvert^2 u]\d t 
    - \tfrac{1}{2}F u\d t 
    - i u \Phi \d W,
\end{equation}
with $F$ being defined as
\begin{equation}
    \label{eq:itocorrection}
    F \coloneqq \sum_{k \in \N} (\Phi e_k)^2,
\end{equation}
where $e_k$ is an orthonormal basis of $L^2(\R;\R)$.
Let us collect some facts about $\Phi$ and $F$ which will be used throughout.
The proof of Proposition~\ref{prop:hs} is contained in Appendix~\ref{app:hs}.
\begin{proposition}
\label{prop:hs}
Let $\phi \in L^2(\R;\R)$ and $u \in L^2_x$. Then the series in \eqref{eq:itocorrection} is well-defined and we have the equalities
\begin{subequations}
\begin{align}
        \label{eq:Fbeta}
        F &= \beta^2, \\
        \label{eq:hilbertschmidteq}
        \lVert u \Phi \rVert_{\mathcal{L}_2(L^2(\R;\R);L^2_x)} &= \beta \lVert u \rVert_{L^2_x}.
\end{align}
If additionally $\phi \in H^s(\R;\R)$ and $u \in H^s_x$ for some $s \in [0,\infty)$, then we have the estimate
\begin{align}
        \label{eq:hilbertschmidtest}
        \lVert u \Phi \rVert_{\mathcal{L}_2(L^2(\R;\R);H^s_x)} &\leq C_{s} \lVert \phi \rVert_{H^s_x} \lVert u \rVert_{H^s_x}
\end{align}
\end{subequations}
for some $C_s > 0$ which depends only $s$.
\end{proposition}
Substituting \eqref{eq:Fbeta} into \eqref{eq:spfnlsitoformal}, the stochastic PFNLS equation in It\^o form reads
\begin{equation}
    \label{eq:spfnlsito}
    \d u = [i \Delta u - i \nu u - \epsilon (\gamma u - \mu \overline{u}) + i \kappa \lvert u \rvert^2 u]\d t - \tfrac{1}{2} \beta^2 u\d t - i u \Phi \d W.
\end{equation}
From the definition of $\Phi$ \eqref{eq:defPhi}, it is clear that this operator commutes with translation.
Furthermore, since $\xi \coloneqq \d W$ formally represents a white noise, its statistics are also invariant under translation.
Thus, the noise terms do not break the temporal- and spatial translation symmetries inherent to \eqref{eq:pfnls} (in a statistical sense).

Before we proceed with the mathematical analysis, we give a meaningful interpretation to our noise.
Since $\xi$ formally has a covariance operator on $L^2(\R;\R)$ equal to the identity, it can be seen using \eqref{eq:defPhi} that $\Phi \xi$ formally satisfies the covariance relation
\begin{align*}
        \mathbb{E}\Bigl[(\Phi \xi)(t,x) \cdot (\Phi\xi)(t',x') \Bigr]
        &=\mathbb{E}\Bigl[\langle \Phi \xi(t), \delta_{x} \rangle_{L^2_x} \langle \Phi \xi(t'), \delta_{x'} \rangle_{L^2_x} \Bigr] \\
        &= \delta_0(t - t')\langle \Phi^* \delta_x, \Phi^*\delta_{x'}\rangle_{L^2_x} \\
        &= \delta_0(t - t')(\tilde{\phi}*\phi)(x - x'),
\end{align*}
where $\delta_a$ denotes a Dirac mass at the point $x = a$, and  $\tilde{\phi}$ denotes the reflection of $\phi$ around the origin.
Therefore, $g \coloneqq \tilde{\phi}*\phi$ can be interpreted as the spatial correlation function of our noise.
Note that $g$ is an even function, so that the correlation only depends on $|x - x'|$.
The variance at any point is given by $g(0) = \beta^2$, which means this quantity can be viewed as the strength of the noise.

\subsection{Strichartz estimates}

In the analysis of nonlinear Schr\"odinger equations, the dispersion displayed by the linear Schr\"odinger equation plays a major role.
In our context, this dispersion manifests in the form of \emph{Strichartz estimates}.
These estimates give control over certain space-time mixed Lebesgue norms of solutions to the linear Schr\"odinger equation.
In our one-dimensional setting, they take the following form.
\begin{definition}
\label{def:admissible}
A pair $(r,p)$ with $r \in [4,\infty]$, $p \in [2,\infty]$ is called \emph{admissible} if it satisfies
\begin{align}
    \label{eq:defadmissible}
    \frac{2}{r} + \frac{1}{p} = \frac{1}{2}.
\end{align}
\end{definition}
\begin{theorem}[Strichartz estimates]
\label{thm:strichartz}
Let $s \in [0,\infty)$, and let $(r,p) \neq (4,\infty)$ and $(\alpha,\delta)$ be admissible.
There exists a constant $C$, such that the estimates
\begin{subequations}
\label{eq:strich}
\begin{align}
    \label{eq:strichhom}
    \lVert S(\cdot)f \rVert_{L^r(0,T;H^{s,p}_x)} &\leq C \lVert f \rVert_{H^{s}_x}, \\
    \label{eq:strichconv}
    \Bigl\lVert \int_0^{\cdot}S(\cdot - t')g(t')\d t' \Bigr\rVert_{L^{r}(0,T;H^{s,p}_x)} &\leq C \lVert g \rVert_{L^{\alpha'}(0,T;H^{s,\delta'}_x)}, \\
    \label{eq:strichstoch}
    \Bigl\lVert \int_0^{\cdot}S(\cdot - t') h(t') \Phi \d W(t') \Bigr\rVert_{L^q_{\Omega}(L^r(0,T;H^{s,p}_x))} &\leq C \sqrt{q} \lVert \phi \rVert_{H^{s}_x} \lVert h \rVert_{L^q_{\Omega}(L^2(0,T;H^{s}_x))},
\end{align}
hold for every $q \in [2,\infty)$, $T \in (0,\infty]$, $f \in H^{s}_x$, $g \in L^{\alpha'}(0,T;H^{s,\delta'}_x)$, $h \in L^q_{\Omega}(L^2(0,T;H^{s}_x))$, and $\phi \in L^2(\R;\R) \cap H^{s}_x$ (recall \eqref{eq:defphi}).
\end{subequations}
\end{theorem}
\begin{remark}
In the case $(r,p) = (\infty,2)$, the relevant processes in Theorem~\ref{thm:strichartz} have continuous versions, and the $L^{\infty}$-norm on the left-hand side of \eqref{eq:strich} can be replaced by $C([0,T])$. We will always use these continuous versions.
This also applies to \eqref{eq:PPstrich} further below.
\end{remark}
\begin{remark}
Estimates \eqref{eq:strichhom} and \eqref{eq:strichconv} still hold in the case $(r,p) = (4,\infty)$.
This also applies to \eqref{eq:PPstrichhom}, \eqref{eq:PPstrichconv}, \eqref{eq:PP0strichhom}, and \eqref{eq:PP0strichconv} further below.
\end{remark}
Estimates \eqref{eq:strichhom}, \eqref{eq:strichconv}, and \eqref{eq:strichstoch} are commonly referred to as the \emph{homogeneous}, \emph{convolution}, and \emph{stochastic} Strichartz estimates respectively.
The homogeneous and convolution Strichartz estimates are well-known and can be found for example in~\cite[Theorem 2.3.3]{cazenave_semilinear_2003} or~\cite{keel_endpoint_1998}.
The stochastic Strichartz estimate is more recent, and was first shown in~\cite{brzezniak_stochastic_2014} for the case $r = q$.
The proof of our formulation of \eqref{eq:strichstoch}, which is contained in Appendix~\ref{app:stochstrich}, follows the same idea as~\cite{brzezniak_stochastic_2014}, except that we use~\cite[Theorem 1.1]{seidler_exponential_2010} to obtain a constant which is $\mathcal{O}(\sqrt{q})$.

\subsection{Solitary waves and linear stability}
\label{sec:detstab}
We now fix a set of parameters $\nu \in \R$, $\epsilon,\gamma,\mu > 0$ which satisfy $\mu > \gamma$.
We additionally fix $\theta \in [0,2 \pi)$ such that $\cos(2\theta) = \frac{\gamma}{\mu}$ and $\sin(2\theta) > 0$.
This ensures that the deterministic equation \eqref{eq:pfnls} has a stable solitary wave solution $u^*$, explicitly given by
\begin{equation}
    \label{eq:ustar}
    u^*(x) = \sqrt{\frac{2 (\nu + \epsilon \mu \sin(2 \theta))}{\kappa}}\text{sech}(\sqrt{\nu + \epsilon \mu \sin(2\theta)}x)e^{i \theta}
\end{equation}
(see~\cite[equation (1.8)]{kapitula_stability_1998}).
We remark that $u^*$ is infinitely often differentiable, and all of its derivatives are rapidly decaying.

We will frequently make use of expansions around the solitary wave $u^*$.
Due to the cubic term in \eqref{eq:spfnls}, this will require expansions of terms like $\lvert a + b \rvert^2 (a + b)$.
Here, the absolute values prevent the use of convenient multinomial expansion formulas.
To remedy this, we introduce the following notation, which we call the \emph{triple bracket}:
\begin{equation}
\label{eq:triple}
\begin{aligned}
    \{\cdot,\cdot,\cdot\} &\colon \mathbb{C} \times \mathbb{C} \times \mathbb{C} \rightarrow \mathbb{C} \\
    \{a,b,c\} &= a b \overline{c} + a \overline{b} c + \overline{a} b c.
\end{aligned}
\end{equation}
Observe that the triple bracket is symmetric, (real-)trilinear and that $\lvert u \rvert^2 u = \tfrac{1}{3}\{u,u,u\}$.
Therefore, we can compactly write binomial expansions like
\begin{align*}
    \lvert u+v \rvert^2 (u+v) = \tfrac{1}{3}\{u + v, u + v, u + v\} = \tfrac{1}{3}\{u,u,u\} + \{u,u,v\} + \{u,v,v\} + \tfrac{1}{3}\{v,v,v\}.
\end{align*}
This notation is particularly useful when using multinomial expansions with more terms.
For readability, we abbreviate
\begin{equation}
    \label{eq:defLstraight}
    L u \coloneqq - i \nu u - \epsilon (\gamma u - \mu \overline{u}).
\end{equation}
Combining our new notation, we may compactly rewrite \eqref{eq:pfnls} as
\begin{align*}
    \partial_t u = i \Delta u + L u + \tfrac{1}{3} i \kappa  \{u,u,u\}.
\end{align*}
Using the additivity of the triple bracket, it is now straightforward to see that the operator
\begin{align}
    \label{eq:defL} 
    \mathcal{L} \colon v &\mapsto i \Delta v + L v + i \kappa \{u^*, u^*, v \}  
\end{align}
corresponds to the linearization of \eqref{eq:pfnls} around the solitary wave $u^*$.
The linear stability of the solitary wave \eqref{eq:ustar} is captured in the following theorem, which has been shown in~\cite{kapitula_stability_1998}.
\begin{theorem}
    \label{thm:Lprops}
    The operator $\mathcal{L}$ has the following properties:
    \begin{enumerate}
    \item $\mathcal{L}$ is the generator of a strongly continuous semigroup on $L^2_x$, denoted by $P(t)$.
    \item $u^*_x$ is an eigenfunction of $\mathcal{L}$ with eigenvalue $0$, which has algebraic multiplicity one.
    \item The spectrum of $\mathcal{L}$ on $L^2_x$ is contained in $\{ z \in \C : \Real (z) \leq -b \} \cup \{0\}$ for some $b > 0$. Thus, the Riesz spectral projection
    \begin{equation*}
        \Pi^0 \coloneqq \frac{1}{2\pi i}\oint_C (\lambda I - \mathcal{L})^{-1}d\lambda,
    \end{equation*}
    is well-defined if $C$ is a sufficiently small contour encircling $0$ counterclockwise.
    \item If we additionally define $\Pi \coloneqq I - \Pi^0$, then there exist constants $M$ and $a > 0$ such that
    \begin{equation}
        \label{eq:Pdecay}
        \lVert P(t) \Pi \rVert_{\mathcal{L}(L^2_x)} \leq M e^{-at}
    \end{equation}
    holds for all $t \in [0,\infty)$.
    \end{enumerate}
\end{theorem}
\begin{remark}
The operator $\mathcal{L}$ is not complex-linear, and the same applies to $P(t)$, $\Pi^0$, and $\Pi$.
Additionally, $\Pi^0$ projects onto the real span of $u^*_x$ as opposed to the complex span.
Thus, in the context of the linearization we should regard $L^2_x \simeq L^2(\R;\R^2)$ as a real vector space.
\end{remark}
Using $\Pi$, we also define the linear operator $\mathcal{P}$ as follows:
\begin{equation}
\label{eq:defP}
    \mathcal{P} \colon f \mapsto \frac{\langle f - \Pi f, u^*_x \rangle_{L^2_x}}{\lVert u^*_x \rVert_{L^2_x}^2}.
\end{equation}
\begin{proposition}
\label{prop:P}
The operator $\mathcal{P}$ is bounded from $L^2_x$ to $\R$, and for every $f \in L^2_x$ we have the decomposition
\begin{equation*}
    f = \Pi f + \mathcal{P}(f) u^*_x.
\end{equation*}
\end{proposition}
\begin{proof}
The boundedness of $\mathcal{P}$ follows from the boundedness of $\Pi$ and the Cauchy--Schwarz inequality.
Now fix $f \in L^2_x$. Since $I = \Pi + \Pi^0$ and $\Pi^0$ projects onto the span of $u^*_x$, there exists a unique $a \in \R$ such that
\begin{equation*}
    f = \Pi f + \Pi^0 f = \Pi f + a u^*_x.
\end{equation*}
Rearranging this equation, taking inner products with $u^*_x$ and dividing by $\lVert u^*_x \rVert_{L^2_x}^2$ shows that $a = \mathcal{P}(f)$.
\end{proof}
We now formulate appropriate Strichartz estimates for the semigroups $P(\cdot)\Pi$ and $P(\cdot) \Pi^0$ separately.
Using the decomposition ${P(t) = P(t) \Pi + P(t) \Pi^0}$, we also obtain Strichartz estimates for $P(t)$.
\begin{proposition}[Strichartz estimates for $P(\cdot)\Pi$]
    \label{prop:strichartzPP}
    Let $(r,p) \neq (4,\infty)$ be admissible. There exists a constant $C$, such that the estimates
    \begin{subequations}
    \label{eq:PPstrich}
    \begin{align}
        \label{eq:PPstrichhom}
        \lVert P(\cdot)\Pi f \rVert_{L^r(0,T;L^p_x)} &\leq C \lVert f \rVert_{L^2_x}, \\
        \label{eq:PPstrichconv}
        \Bigl\lVert \int_0^{\cdot} P(\cdot-t')\Pi g(t')\d t' \Bigr\rVert_{L^r(0,T;L^p_x)} 
        &\leq C \lVert g \rVert_{L^1(0,T;L^2_x)}, \\
        \label{eq:PPstrichstoch}
        \Bigl\lVert \int_0^{\cdot} P(\cdot-t')\Pi h(t')\Phi \d W(t') \Bigr\rVert_{L^q_{\Omega}(L^r(0,T;L^p_x))} 
        &\leq C \sqrt{q} T^{\frac{1}{2} - \frac{1}{q}} \beta \lVert h \rVert_{L^q_{\Omega}(L^{q}(0,T;L^2_x))}, 
    \end{align}
    \end{subequations}
    hold for all $q \in [2,\infty)$, $T \in (0,\infty)$, $f \in L^2_x$, $g \in L^1(0,T;L^2_x)$, $h \in L^q_{\Omega}(L^q(0,T;L^2_x))$, and $\phi \in L^2(\R;\R)$ (recall \eqref{eq:defphi}).
\end{proposition}

\begin{proof}
    We first show \eqref{eq:PPstrichhom}.
    Consider for some $f \in L^2_x$ the evolution equation
    \begin{equation}
        \begin{aligned}
        \label{eq:strichdu}
        \d u &= [i \Delta u - i \nu u - \epsilon(\gamma u - \mu \overline{u}) + i \kappa \{u^*,u^*,u\}]\d t, \\
        u(0) &= \Pi f.
        \end{aligned}
    \end{equation}
    By standard semigroup theory, it can be shown that \eqref{eq:strichdu} has a unique solution \\${u \in C([0,t];L^2_x)}$, which satisfies the following equalities:
    \begin{subequations}
    \label{eq:strichut}
        \begin{align}
            \label{eq:strichutP}
            u(t) &= P(t)\Pi f, \\
            \label{eq:strichutS}
            u(t) &= S(t)\Pi f + \int_0^t S(t-t')\bigl(- i \nu u - \epsilon(\gamma u - \mu \overline{u}) + i \kappa \{u^*,u^*,u\}\bigr) \d t'.
        \end{align}
    \end{subequations}
    Using the decay estimate \eqref{eq:Pdecay} from Theorem~\ref{thm:Lprops}, we first observe that
    \begin{equation}
        \label{eq:strichuL1}
        \lVert u \rVert_{L^1(0,T;L^2_x)} \overset{\eqref{eq:strichutP}}{=} \lVert P(\cdot)\Pi f \rVert_{L^1(0,T;L^2_x)} 
        \overset{\eqref{eq:Pdecay}}{\leq} \lVert f \rVert_{L^2_x} \int_0^T M e^{-at}\d t  \leq M a^{-1} \lVert f \rVert_{L^2_x}.
    \end{equation}
    From \eqref{eq:strichutS} and Theorem~\ref{thm:strichartz}, it now follows that
    \begin{align*}
        \lVert u \rVert_{L^r(0,T;L^p_x)} 
        \overset{\eqref{eq:strichhom},\eqref{eq:strichconv}}&{\leq} C \bigl(\lVert \Pi f \rVert_{L^2_x} 
        + \lVert - i \nu u - \epsilon(\gamma u - \mu \overline{u}) + i \kappa \{u^*,u^*,u\} \rVert_{L^1(0,T;L^2_x)}\bigr) \\
        &  \leq C \lVert \Pi f \rVert_{L^2_x} + C' \lVert u \rVert_{L^1(0,T;L^2_x)}
        \overset{\eqref{eq:strichuL1}}{\leq} C'' \lVert f \rVert_{L^2_x},
    \end{align*}
    which shows \eqref{eq:PPstrichhom}.
    To show \eqref{eq:PPstrichconv}, we use Minkowski's integral inequality and \eqref{eq:PPstrichhom}:
    \begin{align*}
        \Bigl\lVert \int_0^{\cdot} &P(\cdot-t')\Pi g(t')\d t' \Bigr\rVert_{L^r(0,T;L^p_x)} 
        = \Bigl\lVert \int_0^T \mathbbm{1}_{[t',T]}(\cdot) P(\cdot-t')\Pi g(t')\d t' \Bigr\rVert_{L^r(0,T;L^p_x)} \\
        &\leq \int_0^T \lVert \mathbbm{1}_{[t',T]}(\cdot) P(\cdot-t')\Pi g(t') \rVert_{L^r(0,T;L^p_x)} \d t'
        = \int_0^T \lVert P(\cdot)\Pi g(t') \rVert_{L^r(0,T - t';L^p_x)} \d t' \\
        \overset{\eqref{eq:PPstrichhom}}&{\leq} C \int_0^T \lVert g(t') \rVert_{L^2_x} \d t'.
   \end{align*}
   To obtain the stochastic estimate \eqref{eq:PPstrichstoch} for $(r,p) \neq (\infty,2)$, we simply repeat the first part of the proof of \eqref{eq:strichstoch} from Appendix~\ref{app:stochstrich}, replacing all occurences of $S(t)$ with $P(t)\Pi$ and using \eqref{eq:PPstrichhom} instead of \eqref{eq:strichhom} in the intermediate steps. Using H\"older's inequality at the end then gives \eqref{eq:PPstrichstoch}.
   
   For the case $(r,p) = (\infty,2)$, the proof strategy in Appendix~\ref{app:stochstrich} is no longer applicable, since $P(t)\Pi$ is not unitary.
   Instead, we estimate the stochastic convolution using the well-known factorization method (see for instance~\cite[Theorem 4.5]{van_neerven_maximal_2020} for a version applicable to our setting), which gives the result after applying \eqref{eq:hilbertschmidteq}.
\end{proof}
For $P(t)\Pi^0$, there is significantly more freedom in choosing the exponents, and the requirement of admissibility can be dropped.
In this case, the estimates follow not from any dispersive phenomena, but rather from the fact that the range of $\Pi^0$ is one-dimensional, being spanned by $u^*_x$.
\begin{proposition}
\label{prop:strichartzPP0}
Let $p \in [1,\infty]$.
There exists a constant $C$, such that the estimates
    \begin{subequations}
    \label{eq:PP0strich}
    \begin{align}
        \label{eq:PP0strichhom}
        \lVert P(\cdot)\Pi^0 u_0 \rVert_{C([0,T];L^p_x)} &\leq C \lVert u_0\rVert_{L^2_x}, \\
        \label{eq:PP0strichconv}
        \Bigl\lVert \int_0^{\cdot} P(\cdot-t')\Pi^0f(t')\d t' \Bigr\rVert_{C([0,T];L^p_x)} 
        &\leq C \lVert f \rVert_{L^1(0,T;L^2_x)}, \\
        \label{eq:PP0strichstoch}
        \Bigl\lVert \int_0^{\cdot} P(\cdot-t')\Pi^0g(t')\Phi \d W(t') \Bigr\rVert_{L^q_{\Omega}(C([0,T];L^{p}_x))} 
        &\leq C \sqrt{q} \beta \lVert g \rVert_{L^q_{\Omega}(L^2(0,T;L^2_x))}, 
    \end{align}
    \end{subequations}
    hold for all $q \in [2,\infty)$, $T \in (0,\infty]$, $u_0 \in L^2_x$, $f \in L^1(0,T;L^2_x)$, $g \in L^q_{\Omega}(L^2(0,T;L^2_x))$, and $\phi \in L^2(\R;\R)$ (recall \eqref{eq:defphi}).
\end{proposition}
\begin{proof}
Since $\mathcal{L}u^*_x = 0$ by Theorem~\ref{thm:Lprops}, it holds that $P(t)u^*_x = u^*_x$.
After observing that the range of $\Pi^0$ is spanned by $u^*_x$, it follows that $P(t)\Pi^0 = \Pi^0$ for every $t$.
Thus, we get
\begin{align}
    \label{eq:PP0Lpest}
    \lVert P(t)\Pi^0 u_0 \rVert_{L^p_x} = \lVert \Pi^0 u_0 \lVert_{L^p_x} = \frac{\lVert u^*_x \rVert_{L^p_x}}{\lVert u^*_x \rVert_{L^2_x}} \lVert \Pi^0 u_0 \rVert_{L^2_x} \leq C \lVert u_0 \lVert_{L^2_x},
\end{align}
where $\lVert u^*_x\rVert_{L^p_x} < \infty$ because $u^*_x$ decays rapidly.
Using Minkowski's inequality, we can additionally estimate
\begin{align*}
\Bigl\lVert \int_0^{t} P(t-t')\Pi^0f(t')\d t' \Bigr\rVert_{L^p_x}
\leq  \int_0^{t} \lVert P(t-t')\Pi^0f(t') \rVert_{L^p_x}\d t'
\overset{\eqref{eq:PP0Lpest}}{\leq} C \int_0^t \lVert f(t')\rVert_{L^2_x} \d t',
\end{align*}
at which point \eqref{eq:PP0strichconv} follows by taking the supremum over $t \in [0,T]$.
Finally, we estimate
\begin{align*}
    \Bigl\lVert \int_0^{\cdot} P(\cdot-t')\Pi^0g(t')\Phi \d W(t') \Bigr\rVert_{L^q_{\Omega}(C([0,T];L^p_x))}
    &= \Bigl\lVert \int_0^{\cdot} \Pi^0 g(t')\Phi \d W(t') \Bigr\rVert_{L^q_{\Omega}(C([0,T];L^p_x))} \\
    &\leq C \Bigl\lVert \int_0^{\cdot} g(t')\Phi \d W(t') \Bigr\rVert_{L^q_{\Omega}(C([0,T];L^2_x))} \\
    &\leq C' \sqrt{q} \lVert g \Phi \rVert_{L^q_{\Omega}(L^2(0,T;\mathcal{L}_2(L^2(\R;\R);L^2_x)))} \\
    \overset{\eqref{eq:hilbertschmidteq}}&{=} C' \sqrt{q}\beta \lVert g  \rVert_{L^q_{\Omega}(L^2(0,T;L^2_x))},
\end{align*}
where we have used the Burkholder--Davis--Gundy inequality for the penultimate step.
\end{proof}

To get appropriate Gaussian tail bounds, we need the following elementary lemma.
\begin{lemma}
\label{lem:tail}
Let $\xi$ be a nonnegative real-valued random variable which satisfies
\begin{align*}
    \lVert \xi \rVert_{L^p_{\Omega}} \leq C \sqrt{p}
\end{align*}
for all sufficiently large $p < \infty$, where $C$ is independent of $p$.
Then $\xi$ satisfies the Gaussian tail bound
\begin{align*}
    \PP [\xi \geq \lambda] \leq \exp\bigl(-e^{-2} C^{-2} \lambda^2 \bigr)
\end{align*}
for all sufficiently large $\lambda$.
\end{lemma}
\begin{proof}
By Markov's inequality and the assumption on $\xi$, we have
\begin{align*}
    \PP [\xi \geq \lambda] = \PP [ \xi^p \geq \lambda^p] \leq \lambda^{-p} C^p \sqrt{p}^p = (\lambda^{-1}C \sqrt{p})^p
\end{align*}
for $p$ sufficiently large.
Choosing $p = e^{-2} C^{-2} \lambda^2$ (which can be made sufficiently large by increasing $\lambda$) gives the result.
\end{proof}

\section{Main results}
\label{sec:mainresult}
We now state the main results.
Theorem~\ref{thm:wellposed} states the mild well-posedness of \eqref{eq:spfnlsito}.
In Section~\ref{sec:asym} we derive an asymptotic expansion of solutions to \eqref{eq:spfnlsito} around a solitary wave centered at the origin.
The validity of this expansion is stated in Theorem~\ref{thm:asym}.
Next, we introduce and motivate our definition of the phase process in Section~\ref{sec:orbital}.
Theorem~\ref{thm:relaxation} then gives a bound on the fluctuations around the shifted wave, and Proposition~\ref{prop:longterm} and Corollary~\ref{cor:longterm} state the orbital stability.

\subsection{Well-posedness}
\label{sec:wellposed}
Our first main result is the well-posedness of a mild formulation of \eqref{eq:spfnlsito}.
The proof is contained in Section~\ref{sec:proofwellposed}.
\begin{theorem}
\label{thm:wellposed}
Let $\nu,\epsilon,\gamma,\mu,\kappa > 0$, let $u_0$ be an $L^2_x$-valued $\mathcal{F}_0$-measurable random variable, let $T \in (0,\infty)$ and $\phi \in L^2(\R;\R)$. There exists a unique $\mathbb{F}$-adapted process $u \in C([0,T];L^2_x)\cap L^6(0,T;L^6_x)$ satisfying the mild-solution equation
\begin{equation}
    \label{eq:spfnlsmild}
    \begin{aligned}
        u(t) = {} & S(t) u_0 + \int_0^t S(t-t')(-i \nu u(t') -\epsilon (\gamma u(t') - \mu \overline{u}(t')) - \tfrac{1}{2}\beta^2 u(t'))\d t' \\
        & + i \kappa \int_0^t S(t-t')\lvert u(t')\rvert^2 u(t')\d t
        - i \int_0^t S(t-t')u(t')\Phi \d W(t'),
    \end{aligned}
\end{equation}
for every $t \in [0,T]$, $\mathbb{P}$-a.s.
Furthermore, $u \in L^r(0,T;L^p_x)$ for any $(r,p) \neq (4,\infty)$ which satisfies \eqref{eq:defadmissible}, and we have the a priori estimate
\begin{align}
    \label{eq:apriori}
    \lVert u(t) \rVert_{L^2_x} \leq e^{-\epsilon(\gamma - \mu)t} \lVert u_0 \rVert_{L^2_x},
\end{align}
for every $t \in [0,T]$, $\mathbb{P}$-a.s.

If we additionally assume that $\phi \in H^{s}_x$ and $u_0$ takes values in $H^{s}_x$ for some $s \in [0,\infty)$, 
then also $u \in C([0,T];H^s_x) \cap L^r(0,T;H^{s,p}_x)$ for any $(r,p) \neq (4,\infty)$ which satisfies \eqref{eq:defadmissible}.
\end{theorem}

\subsection{Asymptotic expansion}
\label{sec:asym}
From now on, let $\nu,\epsilon,\gamma,\mu,\kappa$, and $u^*$ be as described in Section~\ref{sec:detstab}.
Consider the SPFNLS equation \eqref{eq:spfnlsito}, now written using our notational shorthands (cf.\@ \eqref{eq:defphi}, \eqref{eq:triple}, \eqref{eq:defL}), and including an additional parameter $\sigma > 0$ which controls the strength of the noise:
\begin{equation}
    \label{eq:spfnls1}
    \d u = [i \Delta u + L u+ \tfrac{1}{3}i \kappa \{u,u,u\} - \tfrac{1}{2}\beta^2 \sigma^2 u]\d t - i \sigma u \Phi \d W.
\end{equation}
The first step towards showing orbital stability of the solitary wave is to construct an asymptotic expansion to second order in $\sigma$.
For this we use the following ansatz:
\begin{equation}
    \label{eq:ansatz1}
    u = u^* + \sigma v_1 + \sigma^2 v_2 + z,
\end{equation}
where $z$ should be regarded as being $\mathcal{O}(\sigma^3)$.
To match our ansatz, we supply \eqref{eq:spfnls1} with the initial condition
\begin{equation}
    \label{eq:uic}
    u(0) = u^* + \sigma v_{1,0} + \sigma^2 v_{2,0}.
\end{equation}
By using the additivity of the triple bracket, we see that \eqref{eq:spfnls1} can be rewritten as
\begin{equation}
    \label{eq:dubig}
    \begin{aligned}
    \d u &= [(i \Delta + L) u^* + \tfrac{1}{3}i \kappa\{u^*, u^*, u^* \}]\d t\\
    &+ \sigma \bigl( [(i \Delta + L) v_1 + i \kappa \{u^*,u^*,v_1\} ]\d t - i u^* \Phi \d W \bigr) \\
    &+ \sigma^2 \bigl( \bigl[(i \Delta + L) v_2 + i \kappa \{u^*, u^*, v_2 \} + i \kappa \{u^*, v_1, v_1 \} - \tfrac{1}{2}\beta^2 u^*\bigr]\d t - i v_1 \Phi \d W \bigr) \\
    &\quad+ \bigl[(i \Delta + L) z + i \kappa \{u^*,u^*,z\} + i \kappa R - \tfrac{1}{2}\beta^2(\sigma^3 v_1 + \sigma^4 v_2 + \sigma^2 z) \bigr]\d t \\
    &\quad- i (\sigma^3 v_2 + \sigma z) \Phi \d W,
\end{aligned}
\end{equation}
where we have abbreviated
\begin{equation}
    \label{eq:defR}
    \begin{aligned}
    R &\coloneqq 2\{u^*, \sigma v_1, \sigma^2 v_2 \} + \tfrac{1}{3}\{\sigma v_1, \sigma v_1, \sigma v_1 \} \\ 
    &+ 2\{u^*, \sigma v_1, z \} + \{u^*, \sigma^2 v_2, \sigma^2 v_2 \} + \{\sigma v_1, \sigma v_1, \sigma^2 v_2\} \\
    &+ 2\{u^*, \sigma^2 v_2, z \} + \{\sigma v_1, \sigma v_1, z\} + \{\sigma v_1, \sigma^2 v_2, \sigma^2 v_2\}  \\
    &+ \{u^*, z, z\} + 2\{\sigma v_1, \sigma^2 v_2, z\} + \tfrac{1}{3}\{\sigma^2 v_2, \sigma^2 v_2, \sigma^2 v_2 \} \\
    &+ \{\sigma v_1, z, z\} + \{\sigma^2 v_2, \sigma^2 v_2, z\} \\
    &+ \{\sigma^2 v_2, z, z \} \\
    &+ \tfrac{1}{3}\{z, z, z \}
    \end{aligned}.
\end{equation}
Note that the terms in \eqref{eq:defR} are organized according to their order in $\sigma$, and all terms are $\mathcal{O}(\sigma^3)$.
Taking the differential of \eqref{eq:ansatz1} and using \eqref{eq:defL} and \eqref{eq:dubig}, we see that if $v_1$ and $v_2$ satisfy
\begin{subequations}
\label{eq:dv1v2}
\begin{align}
    \label{eq:dv1}
    \d v_1 &= \mathcal{L}v_1 \d t - i u^* \Phi \d W, \\
    \label{eq:dv2}
    \d v_2 &= [\mathcal{L}v_2 + i \kappa \{u^*,v_1,v_1\} - \tfrac{1}{2}\beta^2 u^*] \d t - i v_1 \Phi \d W, \\
    v_1(0) &= v_{1,0}, \\
    v_2(0) &= v_{2,0},
\end{align}
\end{subequations}
then $z$ satisfies
\begin{subequations}
\label{eq:dz}
\begin{align}
    \d z &= [\mathcal{L}z + i \kappa R - \tfrac{1}{2}\beta^2(\sigma^3 v_1 + \sigma^4 v_2 + \sigma^2 z) ]\d t - i (\sigma^3 v_2 + \sigma z) \Phi \d W, \\
    z(0) &= 0
\end{align}
\end{subequations}
(note that $\d u^* = [(i \Delta + L)u^* + \tfrac{1}{3}\{u^*, u^*, u^*\}] \d t$ always holds, since both sides vanish).
We can now formulate our first main result, which states that on any fixed time interval $[0,T]$, the approximation $u \approx u^* + \sigma v_1 + \sigma^2 v_2$ is accurate to second order in $\sigma$ with high probability, as long as $v_1$ and $v_2$ are not too large.
The proof is contained in Section~\ref{sec:proofasym}.

\begin{theorem}[Asymptotic expansion, second order]
\label{thm:asym}
Let $v_{1,0}$ and $v_{2,0}$ be $\mathcal{F}_0$-measurable and $L^2_x$-valued random variables, and let $u$ be the solution to \eqref{eq:spfnls1} with initial condition \eqref{eq:uic}.
The system \eqref{eq:dv1v2} has a unique mild solution given by:
\begin{subequations}
\label{eq:v1v2}
\begin{align}
    \label{eq:v1}
    v_1(t) &= P(t)v_{1,0} - \int_0^t P(t-t')i u^* \Phi \d W(t'), \\
    \label{eq:v2}
    \begin{split}
    v_2(t) &= P(t)v_{2,0} + \int_0^t P(t-t') (i \kappa \{u^*, v_1, v_1 \} - \tfrac{1}{2}\beta^2 u^*)\d t' \\
        &\qquad- \int_0^t P(t-t')i v_1 \Phi \d W(t').
    \end{split}
\end{align}
\end{subequations}
We have $v_1, v_2 \in C([0,T];L^2_x) \cap L^r(0,T;L^p_x)$ for every $T \in (0,\infty)$ and every admissible pair $(r,p) \neq (4,\infty)$, $\mathbb{P}$-a.s.
With these $v_1$ and $v_2$, we have the asymptotic expansion
\begin{align}
    \label{eq:defz}
    u(t) = u^* + \sigma v_1(t) + \sigma^2 v_2(t) + z(t),
\end{align}
where $z$ satisfies \eqref{eq:dz}.
Furthermore, for every $T \in (0,\infty)$ and every admissible pair $(r,p) \neq (4,\infty)$, there exist strictly positive constants $c_1$, $c_2$, $\eps'$, independent of $v_{1,0}$, $v_{2,0}$, such that for the following stopping times
\begin{subequations}
\label{eq:asymtimes}
\begin{align}
    \label{eq:tauv1}
    \tau_{v_1} &\coloneqq \sup \{ t \in [0,T] : \lVert v_1 \rVert_{L^{\infty}(0,t;L^2_x) \cap L^6(0,t;L^6_x)} \leq \sigma^{-1}\eps \}, \\
    \label{eq:tauv2}
    \tau_{v_2} &\coloneqq \sup \{ t \in [0,T] : \lVert v_2 \rVert_{L^{\infty}(0,t;L^2_x) \cap L^6(0,t;L^6_x)} \leq \sigma^{-2}\eps^2 \}, \\
    \label{eq:tauz}
    \tau_{z} &\coloneqq \sup \{ t \in [0,T] : \lVert z \rVert_{L^{\infty}(0,t;L^2_x) \cap L^r(0,t;L^p_x)} \leq c_1 \eps^3 \},
\end{align}
\end{subequations}
we have the inequality
\begin{align}
    \label{eq:zprobest}
    \mathbb{P}\bigl[\tau_{z} < \min\{\tau_{v_1}, \tau_{v_2} \} \bigr] \leq \exp(-c_2 \eps^2 \sigma^{-2})
\end{align}
for all $\sigma$, $\eps$ which satisfy $0 < \sigma \leq \eps \leq \eps'$.
\end{theorem}
\begin{remark}
It would be sufficient in \eqref{eq:asymtimes} to control $v_1$ and $v_2$ in a slightly weaker norm.
However, the choice of $L^{\infty}(0,t;L^2_x) \cap L^6(0,t;L^6_x)$ permits a more convenient proof, and we will be able to easily control $v_1$ and $v_2$ in this norm due to the Strichartz estimates.
\end{remark}
\begin{remark}
Theorem~\ref{thm:asym} by itself does not imply any orbital stability of the solitary wave. In fact, the deterministic stability result (Theorem~\ref{thm:Lprops}) is not necessary to prove Theorem~\ref{thm:asym} (even though we use it indirectly via Proposition~\ref{prop:strichartzPP}).
\end{remark}
The following theorem is a first-order variant of Theorem~\ref{thm:asym}.
The proof is a strictly simpler version of that of Theorem~\ref{thm:asym}, so we choose to omit it.
\begin{theorem}
\label{thm:asym2}
Consider the setting of Theorem~\ref{thm:asym} with $v_{2,0} = 0$ and define $z'$ via
\begin{align}
    \label{eq:defzprime}
    u(t) \eqqcolon u^* + \sigma v_1(t) + z'(t).
\end{align}
For every $T \in (0,\infty)$ and every admissible pair $(r,p) \neq (4,\infty)$ there exist strictly positive constants $c_1$, $c_2$ and $\eps'$, independent of $v_{1,0}$, such that if we introduce the additional stopping time
\begin{align}
    \label{eq:tauz'}
    \tau_{z'} &\coloneqq \sup \{ t \in [0,T] : \lVert z' \rVert_{L^{\infty}(0,t;L^2_x) \cap L^r(0,t;L^p_x)} \leq c_1 \eps^2 \},
\end{align}
we have the inequality
\begin{align}
    \label{eq:zprimeprobest}
    \mathbb{P}\bigl[\tau_{z'} < \tau_{v_1}\bigr] \leq \exp(- c_2 \eps^2 \sigma^{-2}),
\end{align}
for all $\sigma,\eps$ which satisfy $0 < \sigma \leq \eps \leq \eps'$.
\end{theorem}

\subsection{Orbital stability}
\label{sec:orbital}
Theorem~\ref{thm:asym} implies that on any fixed timescale, we have the expansion ${u = u^* + \sigma v_1 + \sigma^2 v_2 + \mathcal{O}(\sigma^3)}$.
However, from \eqref{eq:v1v2} it can be seen that in general, the processes $v_1$ and $v_2$ grow with time.
To show orbital stability of the solitary wave on long timescales, we need to control this growth.
Therefore, we first decompose $v_1$ and $v_2$ in the following way:
\begin{subequations}
\label{eq:ansatz2}
\begin{align}
    \label{eq:ansatzv1}
    v_1 &= a_1 u^*_x + w_1, \\
    \label{eq:ansatzv2}
    v_2 &= a_2 u^*_x + \tfrac{1}{2}a_1^2 u^*_{xx} + w_2,
\end{align}
\end{subequations}
where $a_1$ and $a_2$ are (real-valued) stochastic processes which we will specify later, at which point \eqref{eq:ansatz2} determines $w_1$ and $w_2$.
Substituting \eqref{eq:ansatz2} into \eqref{eq:ansatz1} and using Theorem~\ref{thm:asym}, we get
\begin{align*}
    u = u^* + \sigma a_1 u^*_x + \sigma^2 a_2 u^*_x + \tfrac{1}{2}\sigma^2 a_1^2 u^*_{xx} + \sigma w_1 + \sigma^2 w_2 + \mathcal{O}(\sigma^3).
\end{align*}
The first four terms on the right-hand side exactly constitute a Taylor expansion of ${u^*(x + \sigma a_1 + \sigma^2 a_2)}$ to second order in $\sigma$, and thus we have
\begin{align}
    \label{eq:ansatz3}
    u = u^*(x + \sigma a_1 + \sigma^2 a_2) + \sigma w_1 + \sigma^2 w_2 + \mathcal{O}(\sigma^3),
\end{align}
still on the same fixed timescale.
We will see that for some particular choice of $a_1$ and $a_2$, the processes $w_1$ and $w_2$ exhibit growth behavior which is much more favorable than that of their counterparts $v_1$ and $v_2$.
This is the statement of Theorem~\ref{thm:relaxation}, which gives explicit expressions for $a_1$ and $a_2$, and characterizes the growth behavior of $w_1$ and $w_2$.
This is made possible by the exponential decay of $P(t)\Pi$ \eqref{eq:Pdecay}, which is essentially the content of the deterministic stability result.

As an example, from \eqref{eq:v1} it is clear that $v_1$ is expected to grow like $\sqrt{t}$ (this can be made rigourous by combining \eqref{eq:ansatzv1}, \eqref{eq:a1def}, and \eqref{eq:w1est}).
On the other hand, from \eqref{eq:w1est} we see that the moments of $w_1$ remain bounded in time.
Thus, the term $a_1 u^*_x$ in \eqref{eq:ansatzv1} fully captures the growth of $v_1$.
Similarly, $v_2$ is expected to grow at a rate of $t^2$, whereas \eqref{eq:w2est} shows that $w_2$ only grows like $t$.

From \eqref{eq:ansatz3} it is then clear that $a_1$ and $a_2$ have an interpretation as the first- and second-order corrections to the phase of the solitary wave.
Additionally, since $\Phi$ and $u^*$ do not depend on $t$ and $\omega$, it can be seen from \eqref{eq:a1def} that $a_1$ is a Brownian motion rescaled by $\lVert \mathcal{P}i u^* \Phi \rVert_{\mathcal{L}_2(L^2(\R;\R);\R)}$ and offset by $\mathcal{P}(v_{1,0})$.
The proofs of Theorem~\ref{thm:relaxation}, Propostion~\ref{prop:longterm}, and Corollary~\ref{cor:longterm} are contained in Section~\ref{sec:prooforbstab}.
\begin{theorem}
    \label{thm:relaxation}
    There exist predictable processes $a_1$, $a_2$, $w_1$, $w_2$, such that \eqref{eq:ansatz2} and the condition
    \begin{align}
        \label{eq:Pi0wicond}
        \Pi^0 w_k = 0, \qquad k \in \{1,2\},
    \end{align}
    both hold.
    The processes $a_1$ and $a_2$ are given by
    \begin{subequations}
    \label{eq:adef}
    \begin{align}
    \label{eq:a1def}
        a_1(t) &= \mathcal{P}\Bigl[v_{1,0} - \int_0^t i u^* \Phi \d W(t') \Bigr], \\
        a_2(t) &= \mathcal{P}\Bigl[v_{2,0} + \int_0^t i \kappa \{u^*, v_1, v_1 \} - \tfrac{1}{2}\beta^2 u^*\d t' 
        - \int_0^t i v_1 \Phi \d W(t') - \tfrac{1}{2}a_1(t)^2 u^*_{xx}\Bigr],
    \end{align}
    \end{subequations}
    and the corresponding $w_1$ and $w_2$ are given by
    \begin{subequations}
    \label{eq:wdef}
    \begin{align}
        \label{eq:w1def}
        w_1 &= P(t)\Pi v_{1,0} - \int_0^t P(t-t')\Pi i u^* \Phi \d W(t') \\
        \label{eq:w2def}
        \begin{split}
        w_2 &= P(t)\Pi v_{2,0} + \int_0^t P(t-t')\Pi (i \kappa \{u^*, v_1, v_1 \} - \tfrac{1}{2}\beta^2 u^*)\d t' \\
        &\qquad- \int_0^t P(t-t')\Pi i v_1 \Phi \d W(t') - \tfrac{1}{2}a_1(t)^2 \Pi u^*_{xx}. 
        \end{split}
    \end{align}
    \end{subequations}
    Finally, there exists a constant $C$, such that the estimates
    \begin{subequations}
    \label{eq:west}
    \begin{align}
        \label{eq:w1est}
        \lVert w_1(t) \rVert_{L^q_{\Omega}(L^2_x)} 
            &\leq C \bigl( e^{-at}\lVert v_{1,0}\rVert_{L^q_{\Omega}(L^2_x)} + \sqrt{q} \beta \min \{t^{\frac{1}{2}}, 1 \}  \bigr), \\
        \label{eq:w2est}
        \lVert w_2(t) \rVert_{L^q_{\Omega}(L^2_x)}
            &\leq C\bigl( e^{-at} \lVert v_{2,0} \rVert_{L^q_{\Omega}(L^2_x)} + \lVert v_{1,0} \rVert_{L^{2q}_{\Omega}(L^2_x)}^2
                    + q  \beta^2 t\bigr),
    \end{align}
    hold for every $q \in [2,\infty)$, $v_{1,0} \in L^{2q}_{\Omega}(L^2_x)$, $v_{2,0} \in L^q_{\Omega}(L^2_x)$, $t \in [0,\infty)$, and $\phi \in L^2(\R;\R)$ (recall \eqref{eq:defphi}).
    \end{subequations}
\end{theorem}
Theorems~\ref{thm:asym2} and~\ref{thm:relaxation} then allow us to show the following proposition.
\begin{proposition}
\label{prop:longterm}
Consider equation \eqref{eq:spfnls1} with initial data $u(0) = u^* + v_0$, where $v_0$ is an $L^2_x$-valued $\mathcal{F}_0$-measurable random variable.
There exist strictly positive constants $T$, $\tilde{c}_1$, $\tilde{c}_2$, $\lambda$, $\eps'$ such that the estimates
\begin{subequations}
\label{eq:longtermests}
\begin{align}
    \label{eq:resetest1}
    \mathbb{P}\bigl[\lVert u(T) - u^*(x+\sigma a_1(T)\rVert_{L^2_x} \geq \tilde{c}_1 \eps\bigr] \leq 4\exp(-\tilde{c}_2 \sigma^{-2}\eps^2), \\
    \label{eq:resetest2}
    \mathbb{P}\bigl[\lVert u(t) - u^*(x + \sigma a_1) \rVert_{L^{\infty}(0,T;L^2_x)} \geq \eps\bigr] \leq 4\exp(-\tilde{c}_2 \sigma^{-2}\eps^2),
\end{align}
\end{subequations}
hold for every $0 < \lambda \sigma \leq \eps \leq \eps'$, and every $v_0$ which satisfies $\lVert v_0 \rVert_{L^2_x} \leq \tilde{c}_1\eps$, $\mathbb{P}$-a.s.
\end{proposition}
From the translation invariance of the equation, it is immediate that the previous proposition also holds if we consider an initial condition of the form ${u(0) = u^*(x+a) + v_0}$ for any $a \in \R$. Thus, by \eqref{eq:resetest1} we are at time $T$ in essentially the same situation as at time $0$ (with high probability).
In this way, we can `chain' Proposition~\ref{prop:longterm} to finally obtain the long-term stability result.
\begin{corollary}
\label{cor:longterm}
Let $T,\tilde{c}_1,\tilde{c}_2,\lambda,\eps'$, and $v_0$ be as in Proposition~\ref{prop:longterm}.
Then the estimate
\begin{align}
    \mathbb{P}\Bigl[\sup_{t \in [0,NT]} \inf_a \lVert u(t) - u^*(x + a) \rVert_{L^2_x} \geq \eps\Bigr] \leq 8 N \exp(-\tilde{c}_2 \sigma^{-2}\eps^2)
\end{align}
holds for every $N \in \N$, $0 < \lambda \sigma \leq \eps \leq \eps'$, and every $v_0$ which satisfies ${\lVert v_0 \rVert_{L^2_x} \leq \tilde{c}_1\eps}$, $\mathbb{P}$-a.s.
\end{corollary}

\section{Proof of well-posedness}
\label{sec:proofwellposed}
\subsection{Local well-posedness}\label{sec:local}
Following the approach of de Bouard and Debussche in~\cite{de_bouard_stochastic_1999,de_bouard_stochastic_2003} and Hornung in~\cite{hornung_nonlinear_2018}, we first establish well-posedness of a modified version of equation \eqref{eq:spfnlsmild} in which the nonlinear term $\lvert u\rvert^2u$ is truncated. The truncation allows us to control the nonlinearity, which is otherwise not Lipschitz continuous.

We now fix $T_0 \in (0,\infty)$, $s \in [0,\infty)$, $\phi \in L^2(\R;\R) \cap H^{s}_x$, and $(r,p) \neq (4,\infty)$ which satisfies \eqref{eq:defadmissible}.
All of these will remain fixed throughout the proof.
For $T \in (0,\infty)$, we also introduce the following spaces:
\begin{subequations}
\label{eq:defXTYT}
\begin{align}
    \label{eq:defXT}
    X_T^{s} &\coloneqq C([0,T];H^{s}_x)\cap L^r(0,T;H^{s,p}_x), \\
    \label{eq:defYT}
    X_T &\coloneqq C([0,T];L^2_x) \cap L^{6}(0,T;L^6_x).
\end{align}
\end{subequations}
We truncate the nonlinearity in the $L^6(0,T;L^6_x)$-norm, formulate a contraction argument in the space $L^2_{\Omega}(X_T)$, and additionally show that the fixed-point iteration maps a ball in $L^2_{\Omega}(X^{s}_T)$ into itself to obtain the additional regularity.
Since the pairs $(r,p)$ and $(\infty,2)$ both satisfy \eqref{eq:defadmissible}, we can freely replace the norms on the left-hand side of \eqref{eq:strich} by the $X_T^{s}$-norm, and will do so throughout.

For $R \geq 1$, let $\theta_R$ be the function which takes the value $1$ on $[0,R]$, interpolates linearly between $1$ and $0$ on $[R,2R]$ and is identically zero on $[2R,\infty)$. Also define
\begin{align*}
    (\Theta_R(u))(t) &\coloneqq \theta_R(\lVert u \rVert_{L^6(0,t;L^6_x)})u(t).
\end{align*}
Notice that $\Theta_R$ preserves adaptedness of $u$.
The truncated mild equation now takes the form
\begin{equation}
\label{eq:truncated}
\begin{aligned}
    u(t) &= S(t)u_0-\int_0^t S(t-t')(i\nu u(t') +\epsilon(\gamma u(t') -\mu \overline{u}(t')) + \tfrac{1}{2}\beta^2 u(t'))\d t' \\ 
    &\qquad+i \kappa \int_0^t S(t-t')(\lvert \Theta_R(u)(t') \rvert^2 \Theta_R(u)(t'))\d t'-i\int_0^t S(t-t')u(t')\Phi \d W(t').
\end{aligned}
\end{equation}
\begin{proposition}[Global well-posedness of truncated equation]
\label{prop:truncated}
For every $\mathcal{F}_0$-measurable $u_0 \in L^2_{\Omega}(H^{s}_x)$, 
there is a unique $u \in L^2_{\Omega}(X_{T_0})$ which satisfies \eqref{eq:truncated} for every $t \in [0,T_0]$, $\mathbb{P}$-a.s.
This solution additionally satisfies $u \in L^2_{\Omega}(X^{s}_{T_0})$.
\end{proposition}
For the proof of Proposition~\ref{prop:truncated}, we take inspiration from the fixed point argument that was applied to the stochastic NLS equation with initial data in $L^2_x$ in~\cite[Proposition 3.1]{de_bouard_stochastic_1999}.
The use of the stochastic Strichartz estimate \eqref{eq:strichstoch}, which was unknown at the time, significantly simplifies the proof.

We first formulate some estimates relating to the nonlinearity $\lvert u \rvert^2 u$ and the truncation $\Theta_R(u)$.
\begin{lemma}
\label{lem:besseltriple}
There exists a constant $C$, such that the estimate
\begin{equation}
        \label{eq:besseltriple}
        \lVert \lvert u \rvert^2 u \rVert_{L^{1}(0,T;H^{s}_x)} \leq C T^{\frac{1}{2}} \lVert u \rVert_{L^6(0,T;H^{s,6}_x)}\lVert u \rVert_{L^{6}(0,T;L^6_x)}^2
\end{equation}
holds for all $T \in (0,\infty)$ and $u \in L^6(0,T;H^{s,6}_x)$.
In the case $s = 0$, we can take $C = 1$.
\end{lemma}

\begin{proof}
Since $\frac{1}{2} = \frac{1}{6} + \frac{1}{3}$, it follows from the Kato--Ponce inequality (see for instance~\cite[Theorem 1.4]{gulisashvili_exact_1996}) that
\begin{equation*}
        \lVert f g h \rVert_{H^{s}_x} 
        \leq C \bigl(\lVert f \rVert_{H^{s,6}_x} \lVert g h \rVert_{L^3_x} 
           + \lVert f \rVert_{L^6_x} \lVert gh \rVert_{H^{s,3}_x} \bigr).
\end{equation*}
Applying H\"olders inequality and the Kato--Ponce inequality once more using $\frac{1}{3} = \frac{1}{6} + \frac{1}{6}$ gives
\begin{equation*}
        \lVert f g h \rVert_{H^{s}_x} 
        \leq C' \bigl(\lVert f \rVert_{H^{s,6}_x} \lVert g \rVert_{L^6_x} \lVert h \rVert_{L^6_x} 
           + \lVert f \rVert_{L^6_x} \lVert g \rVert_{H^{s,6}_x} \lVert h \rVert_{L^6_x}
           + \lVert f \rVert_{L^6_x} \lVert g \rVert_{L^6_x} \lVert h \rVert_{H^{s,6}_x} \bigr).
\end{equation*}
The desired estimate now follows by substituting $f = u(t)$, $g = u(t)$, $h = \overline{u}(t)$, integrating over $t$ and using H\"older's inequality.
In the case $s = 0$, we have $H^{s}_x = L^{2}_x$ isometrically so \eqref{eq:besseltriple} with $C = 1$ follows from H\"older's inequality.
\end{proof}
\begin{lemma}
\label{lem:truncation}
The estimates
\begin{subequations}
\label{eq:ThetaRest}
\begin{align}
    \label{eq:ThetaRsmall}
    \lVert \Theta_R(u) \rVert_{L^6(0,T;L^6_x)} &\leq 2R, \\
    \label{eq:ThetaRlip}
    \lVert \Theta_R(u) - \Theta_R(u') \rVert_{L^6(0,T;L^6_x)} &\leq 5\lVert u - u' \rVert_{L^6(0,T;L^6_x)},
\end{align}
\end{subequations}
hold for all $R \geq 1$, $T \in (0,\infty)$ and $u,u' \in L^6(0,T;L^6_x)$.
\end{lemma}
\begin{proof}
To ease notation, we will write $y(t) = \lVert u \rVert_{L^6(0,t;L^6_x)}$ and $y'(t) = \lVert u' \rVert_{L^6(0,t;L^6_x)}$ throughout the proof.
Notice that $y$ and $y'$ are nondecreasing by definition and continuous by dominated convergence.
Also, by the reverse triangle inequality, we have
\begin{equation*}
\lvert y(t) - y'(t) \rvert \leq \lVert u - u' \rVert_{L^6(0,T;L^6_x)}
\end{equation*}
for every $t \in [0,T]$.
To get \eqref{eq:ThetaRsmall}, set $t_{R} \coloneqq \sup\{t \in [0,T] : y(t) \leq 2R \}$.
Then by construction of $\theta_R$ we have
\begin{align*}
    \lVert \theta_R(y(\cdot)) u(\cdot) \rVert_{L^6(0,T;L^6_x)} 
    = \lVert \theta_R(y(\cdot)) u(\cdot) \rVert_{L^6(0,t_R;L^6_x)}
    \leq \lVert u \rVert_{L^6(0,t_R;L^6_x)} 
    =y(t_R) \leq 2R.
\end{align*}
For the second inequality, observe that by construction of $\theta_R$ we have
\begin{align*}
    \lvert \theta_R(y) - \theta_R(y')\rvert \leq R^{-1} \mathbbm{1}_{[0,2R]}(\min\{y, y'\})\lvert y - y'\rvert \leq R^{-1}\theta_{2R}(\min\{y,y'\})\lvert y - y' \rvert
\end{align*}
for every $y,y' \geq 0$.
Thus, using the triangle inequality we may estimate
\begin{align*}
    \lVert \theta_R(y(t))&u(t) - \theta_R(y'(t))u'(t) \rVert_{L^6_x} \leq \lvert \theta_R(y(t)) - \theta_R(y'(t)) \rvert \lVert u(t)\rVert_{L^6_x} + \lVert u(t) - u'(t) \rVert_{L^6_x} \\
    &\leq R^{-1}\theta_{2R}(\min\{y(t), y'(t) \}) \lvert y(t) - y'(t)\rvert \lVert u(t) \rVert_{L^6_x} + \lVert u(t) - u'(t) \rVert_{L^6_x} \\
    &\leq R^{-1}\theta_{2R}(\min\{y(t), y'(t) \}) \lVert u(t) \rVert_{L^6_x} \lVert u - u'\rVert_{L^6(0,T;L^6_x)} + \lVert u(t) - u'(t) \rVert_{L^6_x}.
\end{align*}
for every $t \in [0,T]$.
By swapping $u$ and $u'$ we can obtain the same estimate but with $\lVert u'(t)\rVert_{L^6_x}$ on the right-hand side instead of $\lVert u(t) \rVert_{L^6_x}$.
Thus, writing
\begin{equation*}
z(t) \coloneqq \min\{\lVert u(t) \rVert_{L^6_x}, \lVert u'(t) \rVert_{L^6_x} \}
\end{equation*}
and taking $L^6(0,T)$-norms, we see that \eqref{eq:ThetaRlip} follows from
\begin{align*}
    \bigl\lVert R^{-1} \theta_{2R}(\min\{y(\cdot), y'(\cdot) \}) & z(\cdot) \bigr\rVert_{L^6(0,T)} 
    \leq R^{-1}\bigl\lVert \theta_{2R}(\lVert z\rVert_{L^6(0,\cdot)}) z(\cdot) \rVert_{L^6(0,T)} 
    \leq 4,
\end{align*}
where the final inequality follows exactly like how we derived \eqref{eq:ThetaRsmall}.
\end{proof}

\begin{lemma}
\label{lem:estimates}
Define the operators
\begin{subequations}
\label{eq:defT0123}
\begin{align}
        \label{eq:defT0}
        (\mathcal{T}_0 u_0)(t) &\coloneqq S(t)u_0, \\
        \label{eq:defT1}
        (\mathcal{T}_1 u)(t) &\coloneqq -\int_0^t S(t-t')(i \nu u(t') + \epsilon(\gamma u(t') -\mu \overline{u}(t')) + \tfrac{1}{2}\beta^2 u(t'))\d t', \\
        \label{eq:defT2}
        (\mathcal{T}_2^R u)(t) &\coloneqq i \kappa\int_0^t S(t-t')\lvert \Theta_R(u) \rvert^2 \Theta_R(u) \d t', \\
        \label{eq:defT3}
        (\mathcal{T}_3 u)(t) &\coloneqq - i \int_0^t S(t-t')u(t')\Phi \d W(t').
\end{align}
\end{subequations}
There exists a constant $C$, such that the inequalities
\begin{subequations}
\label{eq:T0123est}
\begin{align}
        \label{eq:T0small}
        \lVert \mathcal{T}_0 u_0 \rVert_{X_T^{s}} &\leq C \lVert u_0 \rVert_{H^{s}_x}, \\
        \label{eq:T1smalllin}
        \lVert \mathcal{T}_1 u \rVert_{X_T^{s}} &\leq C T \lVert u \rVert_{C([0,T];H^{s}_x)}, \\
        \label{eq:T2small}
        \lVert \mathcal{T}_2^R u \rVert_{X_T^{s}} &\leq C T^{\frac{1}{2}}R^2 \lVert u \rVert_{L^6(0,T;H^{s,6}_x)}, \\ 
        \label{eq:T2contr}
        \lVert \mathcal{T}_2^R u - \mathcal{T}_2^R u'\rVert_{X_T} &\leq C T^{\frac{1}{2}}R^2 \lVert u - u' \rVert_{L^6(0,T;L^6_x)}, \\ 
        \label{eq:T3smalllin}    
        \lVert \mathcal{T}_3 u \rVert_{L^2_{\Omega}(X_T^{s})} &\leq C T^{\frac{1}{2}} \lVert u \rVert_{L^2_{\Omega}(C([0,T];H^{s}_x))},
\end{align}
hold for every $T \in (0,\infty)$, $R \geq 1$, $u_0 \in H^{s}_x$, and predictable $u,u' \in L^2_{\Omega}(X^{s}_T)$.
\end{subequations}{}

\end{lemma}
\begin{proof}
The only estimates which do not directly follow immediately from Theorem~\ref{thm:strichartz} are \eqref{eq:T2small} and \eqref{eq:T2contr}.
For \eqref{eq:T2small}, we use Theorem~\ref{thm:strichartz}, Lemma~\ref{lem:besseltriple} and Lemma~\ref{lem:truncation} to estimate
\begin{align*}
    \lVert \mathcal{T}_2^R u \rVert_{X_T^s} \overset{\eqref{eq:strichconv}}{\leq} C \bigl\lVert \lvert\Theta_R(u)\rvert^2 \Theta_R(u)\bigr\rVert_{L^1(0,T;H^{s}_x)} 
    \overset{\eqref{eq:besseltriple}}&{\leq} C' T^{\frac{1}{2}} \lVert \Theta_R(u)\rVert_{L^6(0,T;L^6_x)}^2\lVert \Theta_R(u)\rVert_{L^6(0,T;H^{s,6}_x)} \\
    \overset{\eqref{eq:ThetaRsmall}}&{\leq} C' T^{\frac{1}{2}} (2R)^2 \lVert u \rVert_{L^6(0,T;H^{s,6}_x)}.
\end{align*}
To derive \eqref{eq:T2contr} we write for convenience $v = \Theta_R(u)$ and $v' = \Theta_R(u')$.
Then, from Lemma~\ref{lem:truncation} we see that both $\lVert v \rVert_{L^6(0,T;L^6_x)}$ and $\lVert v' \rVert_{L^6(0,T;L^6_x)}$ are bounded by $2R$.
Thus, by H\"older's inequality and Lemma~\ref{lem:truncation} we have
\begin{align*}
    \bigl\lVert \lvert v \rvert^2 v - \lvert v' \rvert^2 v'\bigr\rVert_{L^1(0,T;L^2_x)}
    &\leq \lVert (v  - v')v \overline{v} \rVert_{L^1(0,T;L^2_x)} \\
    &\quad+ \lVert v' (v - v')\overline{v} \rVert_{L^1(0,T;L^2_x)} \\
    &\quad+ \lVert v' v' (\overline{v} - \overline{v}')\rVert_{L^1(0,T;L^2_x)} \\
    \overset{\eqref{eq:besseltriple}}&{\leq} 3(2R)^2 T^{\frac{1}{2}}\lVert v - v' \rVert_{L^6(0,T;L^6_x)} \\
    \overset{\eqref{eq:ThetaRlip}}&{\leq}60 R^2 T^{\frac{1}{2}} \lVert u - u'\rVert_{L^6(0,T;L^6_x)}.
\end{align*}
The inequality \eqref{eq:T2contr} now follows straightforwardly by combining the above estimate with \eqref{eq:strichconv}.
\end{proof}
\begin{proof}[Proof of Proposition~\ref{prop:truncated}]
Fix $\mathcal{F}_0$-measurable $u_0 \in L^2_{\Omega}(H^{s}_x)$ and define the operator
\begin{align*}
        \mathcal{T}^R (u)(t) \coloneqq \bigl(\mathcal{T}_0 u_0 + \mathcal{T}_1 u + \mathcal{T}_2^R u + \mathcal{T}_3 u \bigr)(t).
\end{align*}
By combining the estimates from \eqref{eq:T0123est} and using linearity of $\mathcal{T}_1$ and $\mathcal{T}_3$, we obtain the inequality
\begin{align*}
    \lVert \mathcal{T}^R(u) - \mathcal{T}^R(u')\rVert_{L^2_{\Omega}(X_T)} &\leq C(1+R^2) (T^{\frac{1}{2}} + T) \lVert u - u' \rVert_{L^2_{\Omega}(X_T)},
\end{align*}
for some $C$ which does not depend on $R$, $T$, $u$, $u'$, or $u_0$.
From \eqref{eq:T0123est} we can also see that $\mathcal{T}^R$ maps $L^2_{\Omega}(X_T)$ into itself.
Thus, by the contraction-mapping principle, for sufficiently small $T$ (independent of $u_0$), $\mathcal{T}^R$ has a unique fixed point in $L^2_{\Omega}(X_T)$, and this fixed point is exactly the solution to \eqref{eq:truncated} on $[0,T]$.

To get a solution on $[T,2T]$, we notice that $T$ could be chosen independently of $z_0$.
Thus, since $u(T) \in L^2(\Omega;H^{s}_x)$, it is possible to restart the solution at time $T$ with initial value $u(T)$ to get a solution on $[T,2T]$.
Repeating this and patching together the solutions, we obtain a solution on $[0,T_0]$.

It only remains to show the additional regularity of $u$.
To do this, observe that by \eqref{eq:T0123est} we also have
\begin{equation*}
    \lVert \mathcal{T}^R(u) \rVert_{L^2_{\Omega}(X_T^s)} 
    \leq C\bigl(\lVert u_0 \rVert_{L^2_{\Omega}(H^{s}_x)} + (1+R^2)(T^{\frac{1}{2}} + T)\lVert u \rVert_{L^2_{\Omega}(X^s_T)}\bigr),
\end{equation*}
for some $C$ which is independent of $R$, $T$, $u$ and $u_0$.
Thus, for $T$ sufficiently small depending only on $R$, we see that $\mathcal{T}^R$ maps the ball 
\begin{equation*}
B \coloneqq \{ u \in L^2_{\Omega}(X^s_T) : \lVert u \rVert_{L^2_{\Omega}(X^s_T)} \leq 2C\lVert u_0 \rVert_{L^2(H^{s}_x)} \}
\end{equation*}
into itself.
Therefore, by the theorem of Banach--Alaoglu, the fixed-point iteration by which we obtained $u$ has a subsequence which converges weakly in $L^2_{\Omega}(X^s_T)$.
Since this subsequence also converges strongly to $u$ in $L^2_{\Omega}(X_T)$, it follows that $u \in L^2_{\Omega}(X^s_T)$ by uniqueness of limits.
Since $T$ was chosen independently of $u_0$, we may repeat this procedure on the intervals $[T,2T]$ and so on to find that $u \in L^2_{\Omega}(X^s_{T_0})$.
\end{proof}

Let us denote by $u_R$ the unique solution to the truncated equation \eqref{eq:truncated} with radius $R$ given by Proposition~\ref{prop:truncated}.
We define for $R \geq 1$ the stopping time
\begin{equation}
        \label{eq:tau_L6L6R}
        \tau_R \coloneqq \sup \{t \in [0,T_0] : \lVert u_R\rVert_{L^6(0,t;L^6_x)} \leq R \},
\end{equation}
which corresponds to the first time the norm $\lVert u_R\rVert_{L^6(0,t;L^6_x)}$ reaches size $R$, and before this time no truncation takes place.
Two solutions $u_{R_1}$ and $u_{{R_1}}$ should therefore coincide on $[0,\min \{ \tau_{R_1},\tau_{R_1} \} ]$. 
This is stated in the following lemma.
\begin{lemma}
\label{lem:uniqueness}
Let $R_1, R_2 \geq 1$. 
Then the equality $u_{R_1}(t) = u_{R_2}(t)$ holds $\mathbb{P}$-a.s. for every ${t \in [0,\min \{\tau_{R_1},\tau_{R_2} \}]}$.
\end{lemma}
\begin{proof}
Set $\tau \coloneqq \min \{\tau_{R_1},\tau_{R_2} \}$, $R \coloneqq \max\{R_1, R_2\}$ and consider the operator
\begin{equation*}
    \mathcal{T}'(u) \coloneqq \mathbbm{1}_{[0,\tau]} \mathcal{T}_R(\mathbbm{1}_{[0,\tau]}u).
\end{equation*}
Repeating the arguments and the end of the proof of Proposition~\ref{prop:truncated}, we see that $\mathcal{T}'$ has a unique fixed point on $L^2_{\Omega}(X_{T_0})$.
On the other hand, since $u_{R_k}$ is a fixed point of $\mathcal{T}_{R_k}$ we also have
\begin{equation*}
    \mathbbm{1}_{[0,\tau]} u_{R_k} 
    = \mathbbm{1}_{[0,\tau]} \mathcal{T}_{R_k}(u_{R_k})
    = \mathbbm{1}_{[0,\tau]} \mathcal{T}_{R_k}(\mathbbm{1}_{[0,\tau]} u_{R_k})
    = \mathbbm{1}_{[0,\tau]} \mathcal{T}_{R}(\mathbbm{1}_{[0,\tau]} u_{R_k}) 
    = \mathcal{T}'(\mathbbm{1}_{[0,\tau]} u_{R_k})
\end{equation*}
for $k \in \{1,2\}$, showing that $\mathbbm{1}_{[0,\tau]} u_{R_1}$ and $\mathbbm{1}_{[0,\tau]} u_{R_2}$ are both fixed points of $\mathcal{T}'$.
Thus, $\mathbb{P}$-a.s. equality of $u$ and $u'$ on $[0,\tau]$ follows.
\end{proof}

Using the stopping times $\tau_R$ introduced in \eqref{eq:tau_L6L6R}, we now define
\begin{align}
\label{eq:deftaustar}
    \tau^* \coloneqq \sup_{R\geq 1} \ \tau_R.
\end{align}
Let us construct a maximal solution $u$ by setting $u(t) \coloneqq u_R(t)$ on $[0,\tau_R]$ for each $R \geq 1$.
By Lemma~\ref{lem:uniqueness}, this process is well-defined on $[0,\tau^*)$.
We collect our findings about $u$ so far in the following proposition.

\begin{proposition}[Local well-posedness of SPFNLS]
\label{prop:local}
The following statements hold $\mathbb{P}$-a.s.:
\begin{enumerate}
\item $u \in X_t^s$ for every $t \in [0,\tau^*)$,
\item $u$ satisfies \eqref{eq:truncated} for all $t \in [0,\tau^*)$,
\item $\tau^* < T_0$ implies $\lim_{t \nearrow \tau^*} \lVert u(t) \rVert_{L^6(0,t;L^6_x)} = \infty$.
\end{enumerate}
\end{proposition}
\subsection{Blow-up}
We now show that the constructed solution can only fail to exist globally if its $L^2_x$-norm blows up.
\begin{proposition}[Blow-up criterion]
\label{prop:blowup}
The implication
\begin{align*}
        \sup_{t \in [0,\tau^*)} \lVert u \rVert_{C([0,t];L^2_x)} < \infty \implies \sup_{t \in [0,\tau^*)} \lVert u \rVert_{L^6(0,t;L^6_x)} < \infty
\end{align*}
holds, $\mathbb{P}$-a.s.
\end{proposition}
\begin{proof}
Fix some $M \geq 1$, and define the stopping time
\begin{subequations}
\label{eq:defblowuptau}
\begin{align}
        \label{eq:defblowuptau_L2}
        \tau &\coloneqq \sup \{\ t \in [0,\tau^*) : \lVert u \rVert_{C([0,t];L^2_x)} \leq M \ \},
\intertext{as well as a recursive sequence of stopping times according to $\tau_0 = 0$ and}
        \label{eq:defblowuptau_L6L6}
        \tau_{N+1} &\coloneqq \sup \{\ t \in [\tau_N,\tau] : \lVert u \rVert_{L^6(\tau_{N},t;L^6_x)} \leq 3 K M \ \}, \quad N \in \N_0,
\end{align}
\end{subequations}
where $K$ is the constant $C$ from the right-hand side of \eqref{eq:strichhom}.
Additionally, we define the event 
\begin{equation*}
A \coloneqq \{\omega \in \Omega : \tau_{N} < \tau, \, \forall N \in \N_0 \},
\end{equation*}
 and claim that $\mathbb{P}(A) = 0$.
To see this, we start the solution from time $\tau_N$ and get the $\mathbb{P}$-a.s. equality
\begin{align}
        u(t)&=S(t - \tau_N)u(\tau_N)-\int_{\tau_N}^t S(t-t')(i\nu u(t') +\epsilon(\gamma u(t') -\mu \overline{u}(t')) + \tfrac{1}{2}\beta^2 u(t')) \d t' \nonumber\\
    &\qquad+i \kappa \int_{\tau_N}^t S(t-t')(\lvert u(t')\rvert^2z(t'))\d t'-i\int_{\tau_N}^t S(t-t')u(t')\Phi \d W(t'). \nonumber \\
            \label{eq:blowup1_mild}
    &\eqqcolon T_1 + T_2 + T_3 + T_4,
\end{align}
for every $t \in [\tau_N, \tau^*)$.
Since the estimates from Lemma~\ref{lem:estimates} are invariant under time translation and the pair $(6,6)$ is admissible (cf.\@ \eqref{eq:defadmissible}), we see that
\begin{subequations}
\begin{align}
        \label{eq:blowup1_estT1}
        \lVert T_1 \rVert_{L^6(\tau_N,\tau_{N+1};L^6_x)} \overset{\eqref{eq:T0small}}&{\leq} K \lVert u(\tau_N) \rVert_{L^2_x} 
        \overset{\eqref{eq:defblowuptau_L2}}{\leq} K M, \\
        \label{eq:blowup1_estT2}
        \lVert T_2 \rVert_{L^6(\tau_N,\tau_{N+1};L^6_x)} \overset{\eqref{eq:T1smalllin}}&{\leq} C (\tau_{N+1}-\tau_N) \lVert u \rVert_{C([\tau_{N},\tau_{N+1}];L^2_x)} 
        \overset{\eqref{eq:defblowuptau_L2}}{\leq} C (\tau_{N+1}-\tau_N) M.
\end{align}
To estimate $T_3$ we use Theorem~\ref{thm:strichartz} and H\"older's inequality:
\begin{equation}
\label{eq:blowup1_estT3}
\begin{aligned}
    \lVert T_3 \rVert_{L^{6}(\tau_N,\tau_{N+1};L^6_x)} 
    \overset{\eqref{eq:strichconv}}&{\leq}C \lVert \lvert u \rvert^2 u \rVert_{L^{1}(\tau_N,\tau_{N+1};L^2_x)} 
    \leq C (\tau_{N+1} - \tau_N)^{\frac{1}{2}} \lVert u \rVert_{L^6(\tau_N,\tau_{N+1};L^6_x)}^3 \\
    \overset{\eqref{eq:defblowuptau_L6L6}}&{\leq} 27 \, C K^3M^3 (\tau_{N+1} - \tau_N)^{\frac{1}{2}}.
\end{aligned}
\end{equation}
\end{subequations}
Taking the $L^{6}(\tau_N,\tau_{N+1};L^{6}_x)$-norm of \eqref{eq:blowup1_mild} and using the triangle inequality along with \eqref{eq:blowup1_estT1}-\eqref{eq:blowup1_estT3} gives
\begin{equation}
\label{eq:blowup1_pwest}
\begin{aligned}
      \lVert u \rVert_{L^6(\tau_N,\tau_{N+1};L^6_x)} \leq K M &+ C M (\tau_{N+1}-\tau_N) + 27\, C K^3 M^3 (\tau_{N+1}-\tau_N)^{\frac{1}{2}} \\
      &+ \Bigl\lVert \int_{\tau_N}^{\cdot} S(\cdot-t')u(t')\Phi \d W(t') \Bigr\rVert_{L^6(\tau_N,\tau_{N+1};L^{6}_x)}.
\end{aligned}
\end{equation}
From \eqref{eq:defblowuptau_L6L6} it is clear that we must have the equality $\lVert u \rVert_{L^6(\tau_N,\tau_{N+1};L^6_x)} = 3KM$ for every $N$ if $\omega \in A$.
On the other hand, since $\tau_N$ is nondecreasing with $N$ and bounded by $T_0$, the second and third term on the right-hand side of \eqref{eq:blowup1_pwest} converge to zero as $N \to \infty$.
Combining these facts, we see that $\mathbb{P}(A)$ is bounded by the probability that the events
\begin{align*}
        A_N \coloneqq \Bigl\{ \omega \in \Omega : \Bigl\lVert \int_{\tau_N}^{\cdot} S({\cdot}-t')u(t')\Phi \d W(t') \Bigr\rVert_{L^6(\tau_N,\tau_{N+1};L^6_x)} \geq KM \Bigr\}
\end{align*}
occur for infinitely many $N$.
However, using Markov's inequality and Theorem~\ref{thm:strichartz}, we can estimate
\begin{align*}
        K^2M^2\mathbb{P}(A_N) &\leq \mathbb{E}\Bigl[\Bigl\lVert \int_{\tau_N}^{\cdot} S(\cdot - t')u(t')\Phi \d W(t') \Bigr\rVert_{L^6(\tau_N,\tau_{N+1};L^6_x)}^2 \Bigr] \\
        &\leq \mathbb{E}\Bigl[\Bigl\lVert \int_{0}^{\cdot} S(\cdot - t')\mathbbm{1}_{[\tau_{N},\tau_{N+1}]}(t')u(t')\Phi \d W(t') \Bigr\rVert_{L^6(0,T_0;L^6_x)}^2 \Bigr] \\
        \overset{\eqref{eq:strichstoch}}&{\leq}C^2 \mathbb{E}\Bigl[\lVert u \rVert_{L^2(\tau_N,\tau_{N+1};L^2_x)}^2 \Bigr].
\end{align*}
Since
\begin{align*}
        \sum_{N=0}^{\infty}\mathbb{E}\Bigl[\lVert u \rVert_{L^2(\tau_N,\tau_{N+1};L^2_x)}^2 \Bigr]
        \leq \mathbb{E}\Bigl[ \lVert u \rVert_{L^2(0,\tau;L^2_x)}^2 \Bigr] \overset{\eqref{eq:defblowuptau_L2}}{\leq} M^2 T_0 < \infty
\end{align*}
by Fubini's theorem, we see that the probabilities $\mathbb{P}(A_N)$ are summable.
Thus, $\mathbb{P}(A) = 0$ by the Borel--Cantelli lemma.
By definition of $A$, this implies $\sup_{t \in [0,\tau)}\lVert u \rVert_{L^6(0,t;L^6_x)} < \infty$, $\mathbb{P}$-a.s.
Recalling that $M$ was arbitrary, we finish the proof by choosing $M$ larger than $\sup_{t \in [0,\tau^*)}\lVert u \rVert_{C([0,t];L^2_x)}$ (if this quantity is finite) so that $\tau = \tau^*$ by \eqref{eq:defblowuptau_L2}.
\end{proof}

\subsection{Conservation}
\label{sec:conservation}
Having formulated a blow-up criterion in terms of the $L^2_x$-norm, we now show that this norm can be controlled pathwise.
This will yield global well-posedness of \eqref{eq:spfnlsmild} in combination with Proposition~\ref{prop:blowup}.
\begin{proposition}
\label{prop:conservation}
The inequality
\begin{align}
        \label{eq:L2conservation}
        \lVert u(t) \rVert_{L^2_x} \leq e^{-\epsilon (\gamma - \mu)t} \lVert u(0) \rVert_{L^2_x}
\end{align}
holds, $\mathbb{P}$-a.s., for every $t \in [0,\tau^*)$.
\end{proposition}
\begin{proof}
By definition of $\tau^*$ \eqref{eq:deftaustar}, it suffices to show the claim for any $R \geq 1$ and $t \in [0,\tau_R]$.
To do so, we apply to $u_R$ the mild It\^o formula proved by Da Prato, Jentzen and R\"ockner~\cite[Theorem 1]{da_prato_mild_2019} with the functional 
\begin{align*}
        \mathcal{M}(u) \coloneqq \tfrac{1}{2}\lVert u \rVert_{L^2_x}^2,
\end{align*}
which has first and second Fr\'echet derivatives given by
\begin{align*}
        \d \mathcal{M}(u)[h_1] = \Real \langle h_1, u \rangle_{L^2_x}, \qquad\quad \d^2 \mathcal{M}(u)[h_1,h_2] = \Real \langle h_1,h_2 \rangle_{L^2_x}.
\end{align*}
Since $S(t)$ is unitary on $L^2_x$, the equalities
\begin{alignat*}{2}
        \mathcal{M}(S(t)u) &= \tfrac{1}{2}\lVert S(t) u \rVert_{L^2_x}  & &= \tfrac{1}{2}\lVert u \rVert_{L^2_x}, \\
        \d \mathcal{M}(S(t)u)[S(t)h_1] &= \Real \langle S(t)h_1, S(t)u \rangle_{L^2_x} & &= \Real \langle h_1, u \rangle_{L^2_x}, \\
        \d^2 \mathcal{M}(S(t)u)[S(t)h_1,S(t)h_2] &= \Real \langle S(t)h_1,S(t)h_2 \rangle_{L^2_x} & &= \Real \langle h_1,h_2 \rangle_{L^2_x} 
\end{alignat*}
hold for every $t \in \R$ and $u$, $h_1$, $h_2 \in L^2_x$, and thus the mild It\^o formula coincides exactly with the regular It\^o formula, except without the term containing $i\Delta$.
Since additionally $u_R(t) = u(t)$ for all $t \in [0,\tau_R]$ by definition, this gives the $\mathbb{P}$-a.s. equality
\begin{subequations}
\label{eq:itoM}
\begin{align}
        \label{eq:itoMT1}
        \mathcal{M}(u(t)) = \mathcal{M}(u(0)) &+ \Real \int_0^t \langle - i \nu u(t') + i \kappa \lvert u(t') \rvert^2 u(t'), u(t')\rangle_{L^2_x} \d t' \\
        \label{eq:itoMT2}
        &- \Real \int_0^t \langle \epsilon(\gamma u(t') - \mu \overline{u}(t')), u(t')\rangle_{L^2_x} \d t' \\
        \label{eq:itoMT3}
        &- \tfrac{1}{2}\Real \int_0^t \langle \beta^2 u(t'), u(t')\rangle_{L^2_x}\d t' \\
        \label{eq:itoMT4}
        & - \Real \int_0^t \langle i u(t')\Phi \d W(t'), u(t')\rangle_{L^2_x} \\
        \label{eq:itoMT5}
        &+ \tfrac{1}{2}\Real \int_0^t \lVert u(t') \Phi \rVert_{\mathcal{L}_2(L^2(\R;\R);L^2_x)}^2 \d t'.
\end{align}
\end{subequations}
for all $t \in [0,\tau_R]$.
From the fact that $\langle u v, w \rangle_{L^2_x} = \langle v, \overline{u} w \rangle_{L^2_x}$, we see that
\begin{align*}
        \langle- i \nu u(t') + i \kappa \lvert u(t') \rvert^2 u(t'), u(t')\rangle_{L^2_x} 
        =  - i \nu \lVert u(t') \rVert_{L^2_x}^2 +  i \kappa \lVert u(t')  \rVert_{L^4_x}^4.
\end{align*}
Taking the real part shows that the second term on the right-hand side of \eqref{eq:itoMT1} vanishes.
Similarly, we can rewrite
\begin{align*}
\langle i u(t')\Phi \d W(t'), u(t')\rangle_{L^2_x} = i \langle \Phi \d W(t'), \lvert u(t') \rvert^2 \rangle_{L^2_x}.
\end{align*}
Since $W(t')$ and $\phi$ (recall \eqref{eq:defphi}) are both real-valued, the inner product on the right-hand side always results in a real scalar.
Thus, \eqref{eq:itoMT4} also vanishes.
Finally, from Proposition~\ref{prop:hs} we see that
\begin{equation*}
\lVert u(t') \Phi \rVert_{\mathcal{L}_2(L^2(\R;\R);L^2_x)}^2 
\overset{\eqref{eq:hilbertschmidteq}}{=} \beta^2 \lVert u(t') \rVert_{L^2_x}^2
= \langle \beta^2 u(t'), u(t') \rangle_{L^2_x},
\end{equation*}
so that \eqref{eq:itoMT3} and \eqref{eq:itoMT5} cancel exactly.
Combining all this, \eqref{eq:itoM} simplifies to
\begin{align*}
        \mathcal{M}(u(t)) &= \mathcal{M}(u(0)) - \Real \int_0^t \langle \epsilon(\gamma u(t') - \mu \overline{u}(t')), u(t')\rangle_{L^2_x} \d t' \\
        &= \mathcal{M}(u(0)) - \epsilon \int_0^t \gamma \lVert u(t') \rVert_{L^2_x}^2 - \mu \Real \langle \overline{u}(t'), u(t')\rangle_{L^2_x} \d t'.
\end{align*}
Applying the Cauchy--Schwarz inequality allows us to deduce
\begin{align*}
        \lVert u(t) \rVert_{L^2_x}^2 \leq \lVert u(0) \rVert_{L^2_x}^2 - 2 \epsilon \int_0^t (\gamma - \mu) \lVert u(t') \rVert_{L^2_x}^2 \d t',
\end{align*}
which implies \eqref{eq:L2conservation} after using Gr\"onwall's lemma and taking square roots.
\end{proof}
\begin{proof}[Proof of Theorem~\ref{thm:wellposed}]
From \eqref{eq:L2conservation} it is immediate that $\mathbb{P}\bigl[\sup_{t \in [0,\tau^*)} \lVert u(t) \rVert_{L^2} = \infty\bigr] = 0$.
Thus, by Proposition~\ref{prop:blowup} the solutions constructed in Proposition~\ref{prop:local} exist on the entire interval $[0,T_0]$, $\mathbb{P}$-a.s.
It only remains to lift the assumption that $u_0 \in L^2_{\Omega}$.
This can be done by considering the initial conditions $u_0^M = \mathbbm{1}_{\lVert u_0 \rVert_{L^2_x} \leq M} u_0$ and taking $M$ to infinity, using pathwise uniqueness to patch together the solutions.
Since this is a well-known standard procedure, we will not elaborate.
\end{proof}

\section{Proof of stability}
\label{sec:proofstability}
\subsection{Asymptotic expansion}
\label{sec:proofasym}
\begin{proof}[Proof of Theorem~\ref{thm:asym}]
Throughout the proof, we will use the notation $A \lesssim B$ to denote that there exists a constant $C$, independent of $v_1$, $v_2$, $\eps$, $\sigma$, and $c_1$, such that $A \leq C B$.

Fix $T \in (0,\infty)$ and an admissible pair $(r,p)$ with $p \in [6,\infty)$.
If we prove the theorem for such $p$, it follows from an iterated application of H\"older's inequality that the theorem also holds for admissible pairs with $p \in [2,6)$, so the restriction on $p$ does not entail any loss of generality.

The existence and uniqeness of the mild solution $v_1 \in C([0,T];L^2_x)$ to \eqref{eq:dv1} follows from standard theory (see for example~\cite[Theorem 5.4]{da_prato_stochastic_1992}.
Using \eqref{eq:PPstrichhom}, \eqref{eq:PPstrichstoch}, and \eqref{eq:PP0strichhom}, \eqref{eq:PP0strichstoch} of Propositions~\ref{prop:strichartzPP} and~\ref{prop:strichartzPP0}, we obtain from \eqref{eq:v1} that $v_1 \in L^r(0,T;L^p_x)$, so that also $v_1 \in L^6(0,T;L^6_x)$, $\mathbb{P}$-a.s.
Combining this with H\"older's inequality shows
\begin{align*}
    \lVert \{ u^*, v_1, v_1 \}\rVert_{L^1(0,T;L^2_x)} 
    \leq 3 T^{\frac{1}{2}} \lVert u^* \rVert_{L^6(0,T;L^6_x)} \lVert v_1 \rVert_{L^6(0,T;L^6_x)}^2.
\end{align*}
By a standard localization procedure we can also get integrability in $\omega$, so that the terms on the right-hand side of \eqref{eq:v2} are well-defined and this is indeed the unique solution for $v_2$.
Again, $v_2 \in L^r(0,T;L^p_x)$ by Propositions~\ref{prop:strichartzPP} and~\ref{prop:strichartzPP0}.

From the definition $z(t) \coloneqq u^* - \sigma v_1(t) - \sigma^2 v_2(t)$, it follows that $z$ satisfies \eqref{eq:dz} in the mild sense, meaning for every $t \in [0,T]$ we have the $\mathbb{P}$-a.s. equality
\begin{equation}
    \label{eq:zmild}
    \begin{aligned}
    z(t) &= \int_0^t P(t-t')i \kappa R(t')\d t' - \tfrac{1}{2}\beta^2 \int_0^t P(t-t')(\sigma^3 v_1 +\sigma^4v_2 +\sigma^2 z)\d t' \\
    & \ - \int_0^t P(t-t')i( \sigma^3v_2 + \sigma z)\Phi \d W(t') \eqqcolon T_1 + T_2 + T_3.
    \end{aligned}
\end{equation}
To show \eqref{eq:zprobest} we define the stopping time $\tau \coloneqq \min\{\tau_{v_1}, \tau_{v_2}, \tau_z\}$, 
and notice that
\begin{equation*}
 \mathbb{P}[\tau_z < \min\{\tau_{v_1},\tau_{v_2} \}] = \mathbb{P}[\tau_z < T, \tau_z \leq \tau].
 \end{equation*}
To estimate the latter probability, we first estimate $T_1$ and $T_2$ on the interval $[0,\tau]$.
We assume $\eps' \leq 1$ and $c_1 \geq 1$, so that we can estimate $\sigma^m \leq \eps^m \leq 1$ for any $m \geq 0$.
We will use this frequently and without further mention.

To estimate $T_2$, note that by Propositions~\ref{prop:strichartzPP} and~\ref{prop:strichartzPP0} we have
\begin{equation}
\label{eq:T2est}
    \lVert T_2 \rVert_{L^{\infty}(0,\tau;L^2_x) \cap L^r(0,\tau;L^p_x)}
    \overset{\eqref{eq:PPstrichconv},\eqref{eq:PP0strichconv}}{\lesssim} \lVert \sigma^3 v_1 + \sigma^4 v_2 + \sigma^2 z \rVert_{L^1(0,\tau;L^2_x)} 
    \overset{\eqref{eq:asymtimes}}{\lesssim} \eps^3.
\end{equation}
Using Propositions~\ref{prop:strichartzPP} and~\ref{prop:strichartzPP0} again, carefully inspecting every term in \eqref{eq:defR} and using H\"older's inequality on the triple bracket, we see that we can also estimate
\begin{equation}
    \label{eq:T1est}
    \lVert T_1 \rVert_{L^{\infty}(0,\tau;L^2_x) \cap L^r(0,\tau;L^p_x)} 
    \overset{\eqref{eq:PPstrichconv},\eqref{eq:PP0strichconv}}{\lesssim} \lVert R \rVert_{L^1(0,\tau;L^2_x)} 
    \overset{\eqref{eq:asymtimes}}{\lesssim} \eps^3 + c_1^3 \eps^4.
\end{equation}
Combining \eqref{eq:zmild}, \eqref{eq:T2est}, and \eqref{eq:T1est} with the triangle inequality we get the estimate
\begin{equation}
    \label{eq:zest}
    \lVert z \rVert_{L^{\infty}(0,\tau;L^2_x) \cap L^r(0,\tau;L^p_x)} \leq C (\eps^3+c_1^3 \eps^4) + \lVert T_3 \rVert_{L^{\infty}(0,\tau;L^2_x) \cap L^r(0,\tau;L^p_x)}
\end{equation}
for some constant $C$ which is independent of $v_1$, $v_2$, $\eps$, $\sigma$, and $c_1$.
This allows us to set $c_1 = 4C$ and $\eps' = c_1^{-3}$.
Suppose now that $\tau_z < T$ and $\tau_z \leq \tau$.
Then since ${z \in C([0,T];L^2_x) \cap L^r(0,T;L^p_x)}$, we have by continuity:
\begin{equation*}
    c_1 \eps^3 
    \overset{\eqref{eq:tauz}}{=} \lVert z \rVert_{L^{\infty}(0,\tau_z;L^2_x) \cap L^r(0,\tau_z;L^p_x)} 
    \overset{\eqref{eq:zest}}{\leq} \tfrac{1}{2}c_1 \eps^3 + \lVert T_3 \rVert_{L^{\infty}(0,\tau;L^2_x) \cap L^r(0,\tau;L^p_x)}.
\end{equation*}
Since this can only happen if $T_3$ is sufficiently large, we can now estimate
\begin{align*}
    \PP [\tau_z < T,\tau_z \leq \tau] 
    &\leq \PP \bigl[\lVert \eps^{-3} T_3 \rVert_{L^{\infty}(0,\tau;L^2_x) \cap L^r(0,\tau;L^p_x)} \geq \tfrac{1}{2}c_1\bigr] \\
    &= \PP \bigl[\lVert \sigma^{-1} \eps^{-2} T_3 \rVert_{L^{\infty}(0,\tau;L^2_x) \cap L^r(0,\tau;L^p_x)} \geq \tfrac{1}{2}c_1 \sigma^{-1} \eps  \bigr].
\end{align*}
It only remains to estimate the latter probability.
We note that for $t \leq \tau$ we have the equality
\begin{equation*}
    \sigma^{-1}\eps^{-2} T_3(t) \overset{\eqref{eq:zmild}}{=} -\int_0^t P(t-t')(\mathbbm{1}_{[0,\tau]}(t')i(\sigma^2\eps^{-2}v_2(t') + \eps^{-2} z(t')))\Phi \d W(t').
\end{equation*}
After estimating the integrand as
\begin{align*}
    \lVert \mathbbm{1}_{[0,\tau]}(t')(\sigma^2\eps^{-2}v_2(t') + \eps^{-2} z(t')) \rVert_{L^{\infty}_{\Omega}(L^2(0,T;L^2_x))} 
    \overset{\eqref{eq:asymtimes}}{\leq} 1 + c_1 \eps \leq 2,
\end{align*}
it follows from \eqref{eq:PPstrichstoch}, \eqref{eq:PP0strichstoch} and Lemma~\ref{lem:tail} that the Gaussian tail estimate
\begin{align*}
    \PP [\tau_z < T,\tau_z \leq \tau] \leq \exp(-c_2 c_1^2 \sigma^{-2} \eps^2 ),
\end{align*}
holds for some $c_2 > 0$ which is independent of $\eps$, $\sigma$, $c_1$, as long as $c_1 \sigma^{-1} \eps$ is sufficiently large.
But since $\eps \sigma^{-1} \geq 1$, this can be accomplished by re-choosing $c_1$ to be larger than before if necessary (and also re-choosing $\eps' = c_1^{-3}$).
\end{proof}
\subsection{Orbital stability}
\label{sec:prooforbstab}
Before we prove Theorem~\ref{thm:relaxation}, we isolate some convolution estimates which are used multiple times in the proof.
These estimates essentially follow from Young's convolution inequality and the exponential decay of $P(t)\Pi$ (which we have not used before this point).
\begin{lemma}
\label{lem:youngdecay}
Let $r \in [1,\infty]$. There exists a constant $C$, such that the estimates
\begin{subequations}
\begin{align}
    \label{eq:detdecayconvest1}
    \Bigl\lVert \int_0^{\cdot} P(\cdot-t')\Pi f(t') \d t' \Bigr\rVert_{L^{\infty}(0,T;L^q_{\Omega}(L^2_x))} &\leq C\min \{T^{\frac{1}{r'}}, 1\} \lVert f \rVert_{L^r(0,T;L^q_{\Omega}(L^2_x))}, \\
    \label{eq:detdecayconvest2}
    \Bigl\lVert \int_0^{\cdot} P(\cdot-t')\Pi g(t') \d t' \Bigr\rVert_{L^{\infty}(0,T;L^q_{\Omega}(L^2_x))} &\leq C\min \{T^{\frac{1}{r'}}, 1\} \lVert g \rVert_{L^q_{\Omega}(L^r(0,T;L^2_x))}, \\
    \label{eq:stochdecayconvest}
    \Bigl\lVert \int_0^{\cdot} P(\cdot-t')\Pi h(t') \Phi \d W(t') \Bigr\rVert_{L^{\infty}(0,T;L^q_{\Omega}(L^2_x))} &\leq C \sqrt{q}\beta \min \{T^{\frac{1}{2}},1 \} \lVert h \rVert_{L^{\infty}(0,T;L^q_{\Omega}(L^2_x))},
\end{align}
\end{subequations}
hold for any $q \in [2,\infty)$, $T \in (0,\infty)$, $f \in L^r(0,T;L^q_{\Omega}(L^2_x))$, $g \in L^q_{\Omega}(L^r(0,T;L^2_x))$, ${h \in L^{\infty}(0,T;L^q_{\Omega}(L^2_x))}$, and $\phi \in L^2(\R;\R)$ (recall \eqref{eq:defphi}).
\end{lemma}
\begin{proof}
First we compute
\begin{align}
    \label{eq:alpharT}
    \alpha_r(T) \coloneqq 
    \lVert P(\cdot)\Pi \rVert_{L^r(0,T;\mathcal{L}(L^2_x))} 
    \overset{\eqref{eq:Pdecay}}{\leq} \lVert M\exp(-a \cdot) \rVert_{L^r(0,T)} 
    \leq C \min \{T^{\frac{1}{r}}, 1 \},
\end{align}
for some $C$ which does not depend on $T$.
It then follows from Young's convolution inequality that
\begin{align*}
    \Bigl\lVert \int_0^{\cdot} P(\cdot-t')\Pi f(t') \d t' \Bigr\rVert_{L^{\infty}(0,T;L^q_{\Omega}(L^2_x))} 
    \leq \alpha_{r'}(T) \lVert f \rVert_{L^r(0,T;L^q_{\Omega}(L^2_x))},
\end{align*}
and also 
\begin{align*}
    \Bigl\lVert \int_0^{\cdot} P(\cdot-t')\Pi g(t') \d t' \Bigr\rVert_{L^{\infty}(0,T;L^q_{\Omega}(L^2_x))}
    &\leq \Bigl\lVert \int_0^{\cdot} P(\cdot-t')\Pi g(t') \d t' \Bigr\rVert_{L^q_{\Omega}(L^{\infty}(0,T;L^2_x))} \\
    &\leq \alpha_{r'}(T) \lVert g \rVert_{L^q_{\Omega}(L^r(0,T;L^2_x))},
\end{align*}
which in combination with \eqref{eq:alpharT} shows \eqref{eq:detdecayconvest1} and \eqref{eq:detdecayconvest2}.
Finally, for $t \in [0,T]$ we estimate
\begin{align*}
    \Bigl\lVert \int_0^t P(t-t')\Pi h \Phi \d W(t') \Bigr\rVert_{L^q_{\Omega}(L^2_x)}
    &\leq C \sqrt{q} \lVert P(t - \cdot) \Pi h(\cdot) \Phi \rVert_{L^q_{\Omega}(L^2(0,t;\mathcal{L}_2(L^2(\R;\R);L^2_x)))} \\
    \overset{\eqref{eq:hilbertschmidteq}}&{=} C \sqrt{q} \beta \lVert P(t - \cdot) \Pi h(\cdot) \rVert_{L^q_{\Omega}(L^2(0,t;L^2_x))} \\
    &\leq C \sqrt{q} \beta \lVert P(t - \cdot) \Pi h(\cdot) \rVert_{L^2(0,t;L^q_{\Omega}(L^2_x))} \\
    &\leq C \sqrt{q}\beta \lVert P(t - \cdot)\Pi \rVert_{L^2(0,t;\mathcal{L}(L^2_x))}\lVert h \rVert_{L^{\infty}(0,T;L^q_{\Omega}(L^2_x))} \\
    &= C \sqrt{q}\beta \alpha_2(T)  \lVert h \rVert_{L^{\infty}(0,T;L^q_{\Omega}(L^2_x))},
\end{align*}
where we have used~\cite[Theorem 1.1]{seidler_exponential_2010} for the first inequality, and the fact that $q \geq 2$ for the third inequality.
Taking the supremum over $t \in [0,T]$ and using \eqref{eq:alpharT} gives \eqref{eq:stochdecayconvest}.
\end{proof}
\begin{proof}[Proof of Theorem~\ref{thm:relaxation}]
    From Proposition~\ref{prop:P} we obtain 
    \begin{align*}
        v_1 &= \mathcal{P} [v_1]u^*_x + \Pi v_1, \\
        v_2 &= \mathcal{P} \bigl[v_2 - \tfrac{1}{2}\mathcal{P}[v_1]^2u^*_{xx}\bigr]u^*_x + \tfrac{1}{2}\mathcal{P}[v_1]^2u^*_{xx}+ \Pi \bigl(v_2 - \tfrac{1}{2}\mathcal{P}[v_1]^2u^*_{xx}\bigr).
    \end{align*}
    If we define
    \begin{align*}
        a_1 &\coloneqq \mathcal{P} [v_1], & w_1 &\coloneqq \Pi v_1, \\
        a_2 &\coloneqq \mathcal{P} \bigl[v_2 - \tfrac{1}{2}\mathcal{P}[v_1]^2u^*_{xx}\bigr],
        & w_2 &\coloneqq \Pi \bigl(v_2 - \tfrac{1}{2}\mathcal{P}[v_1]^2u^*_{xx}\bigr),
    \end{align*}
    then \eqref{eq:ansatz2} and \eqref{eq:Pi0wicond} hold.
    Equations \eqref{eq:adef} and \eqref{eq:wdef} follow by substitution using \eqref{eq:v1v2} and noting that $\Pi$ commutes with $P(t)$.
    
    We will now show \eqref{eq:west}. Throughout the proof, $A \lesssim B$ means that there exists a constant $C$, independent of $v_{1,0}$, $v_{2,0}$, $t$, $q$ and $\phi$ (recall \eqref{eq:defphi}) such that $A \leq C B$.
    We first estimate $w_1$ as follows:
    \begin{align*}
        \lVert w_1(t) \rVert_{L^q_{\Omega}(L^2_x)} 
        \overset{\eqref{eq:w1def}}&{\leq} \lVert P(t)\Pi v_{1,0} \rVert_{L^q_{\Omega}(L^2_x)} + \Bigl\lVert \int_0^t P(t-t')\Pi i u^* \Phi \d W(t') \Bigr\rVert_{L^q_{\Omega}(L^2_x)} \\
        \overset{\eqref{eq:Pdecay},\eqref{eq:stochdecayconvest}}&{\lesssim} 
        e^{-at} \lVert v_{1,0} \rVert_{L^q_{\Omega}(L^2_x)} + \sqrt{q}\beta \min\{t^{\frac{1}{2}}, 1 \},
    \end{align*}
    which is \eqref{eq:w1est}.
    In order to show \eqref{eq:w2est}, we will need two intermediate estimates. 
    Firstly, by Proposition~\ref{prop:strichartzPP} we have
    \begin{equation}
    \label{eq:w1strichest}
    \begin{aligned}
        \lVert w_1 \rVert_{L^q_{\Omega}(L^6(0,t;L^6_x))}
        &\leq \lVert P(\cdot)\Pi v_{1,0} \rVert_{L^q_{\Omega}(L^6(0,t;L^6_x))} + \Bigl\lVert \int_0^{\cdot} P({\cdot}-t')\Pi i u^* \Phi \d W(t')\Bigr\rVert_{L^q_{\Omega}(L^6(0,t;L^6_x))} \\
        \overset{\eqref{eq:PPstrichhom},\eqref{eq:PPstrichstoch}}&{\lesssim} \lVert v_{1,0} \rVert_{L^q_{\Omega}(L^2_x)} + \sqrt{q} \beta \lVert u^* \rVert_{L^q_{\Omega}(L^2(0,t;L^2_x))} \\
        &= \lVert v_{1,0} \rVert_{L^q_{\Omega}(L^2_x)} + \sqrt{q}\beta  t^{\frac{1}{2}}.
    \end{aligned}
    \end{equation}
    It also follows from~\cite[Theorem 1.1]{seidler_exponential_2010} that
    \begin{equation}
        \label{eq:a1est}
        \lVert a_1(t) \rVert_{L^q_{\Omega}} \overset{\eqref{eq:a1def}}{\lesssim} \lVert v_{1,0}\rVert_{L^q_{\Omega}(L^2_x)} + \sqrt{q} \beta t^{\frac{1}{2}}. 
    \end{equation}
    Now we have all the ingredients needed to estimate $w_2$.
    We first replace the occurrences of $v_1$ in \eqref{eq:w2def} by $w_1 + a_1 u^*_x$, in accordance with \eqref{eq:ansatz2}.
    This results in the equality
    \begin{subequations}
    \label{eq:w2decomp}
    \begin{align*}
        w_2(t) &= P(t)\Pi v_{2,0} \\
        &+ \int_0^t P(t-t')\Pi i \kappa \{u^*, w_1, w_1 \}\d t' \\
        &+ 2\int_0^t P(t-t') \Pi i \kappa a_1 \{u^*, u^*_x, w_1 \}\d t' \\
        &+  \int_0^t P(t-t') \Pi i \kappa a_1^2 \{u^*, u^*_x,u^*_x \}\d t' \\
        &-  \tfrac{1}{2}\int_0^t P(t-t') \Pi \beta^2u^*\d t' \\
        &- \int_0^t P(t-t')\Pi i w_1 \Phi \d W(t') \\
        &- \int_0^t P(t-t')\Pi i a_1 u^*_x \Phi \d W(t') \\
        & - \tfrac{1}{2}a_1^2 \Pi u^*_{xx}.
    \end{align*}
    \end{subequations}
    We estimate the $L^q_{\Omega}(L^2_x)$-norm of each term separately, which will show \eqref{eq:w2est}. First, we have
    \begin{align*}
        \lVert P(t)\Pi v_{2,0} \rVert_{L^q_{\Omega}(L^2_x)} \overset{\eqref{eq:Pdecay}}&{\lesssim} e^{-at} \lVert v_{2,0}\rVert_{L^q_{\Omega}(L^2_x)}, \\
        \lVert a_1(t)^2 \Pi u^*_{xx} \rVert_{L^q_{\Omega}(L^2_x)} &\lesssim \lVert a_1(t)^2 \rVert_{L^{q}_{\Omega}} = \lVert a_1(t) \rVert_{L^{2q}_{\Omega}}^2 
        \overset{\eqref{eq:a1est}}{\lesssim} \lVert v_{1,0}\rVert_{L^{2q}_{\Omega}(L^2_x)}^2 + q \beta^2 t.
    \end{align*}
    Next, we use our first intermediate estimate on the term which is quadratic in $w_1$.
    \begin{align*}
        \Bigl\lVert \int_0^t P(t-t')\Pi i \kappa \{u^*, w_1, w_1 \}\d t' \Bigr\rVert_{L^q_{\Omega}(L^2_x)}
        \overset{\eqref{eq:detdecayconvest2}}&{\lesssim} \lVert \{ u^*, w_1, w_1 \} \lVert_{L^{q}_{\Omega}(L^3(0,t;L^2_x))} \\
        \lesssim \lVert u^* \rVert_{L^{\infty}(0,t;L^6_x)} \lVert w_1 \rVert_{L^{2q}_{\Omega}(L^6(0,t;L^6_x))}^2
        \overset{\eqref{eq:w1strichest}}&{\lesssim} \lVert v_{1,0} \rVert_{L^{2q}_{\Omega}(L^2_x)}^2 + q \beta^2 t,
    \end{align*}
    where we have used H\"older's inequality for the second step.
    We also estimate
    \begin{align*}
        \Bigl\lVert \int_0^t P(t-t')\Pi i \kappa a_1 \{u^*, u^*_x, w_1 \}\d t' \Bigr\rVert_{L^q_{\Omega}(L^2_x)}
        \overset{\eqref{eq:detdecayconvest1}}&{\lesssim} \lVert a_1 \{ u^*, u^*_x, w_1 \} \rVert_{L^{\infty}(0,t;L^q_{\Omega}(L^2_x))} \\
        \lesssim \lVert a_1 \rVert_{L^{\infty}(0,t;L^{2q}_{\Omega})} \lVert w_1 \rVert_{L^{\infty}(0,t;L^{2q}_{\Omega}(L^2_x))}
        \overset{\eqref{eq:w1est},\eqref{eq:a1est}}&{\lesssim} \lVert v_{1,0} \rVert_{L^{2q}_{\Omega}(L^2_x)}^2 + q \beta^2 t,
    \intertext{as well as}
        \Bigl\lVert \int_0^t P(t-t')\Pi i \kappa a_1^2\{u^*, u^*_x, u^*_x\}\d t' \Bigr\rVert_{L^{q}_{\Omega}(L^2_x)}
        \overset{\eqref{eq:detdecayconvest1}}&{\lesssim} \lVert a_1^2 \{ u^*, u^*_x, u^*_x \} \rVert_{L^{\infty}(0,t;L^q_{\Omega}(L^2_x))} \\
        \lesssim \lVert a_1 \rVert_{L^{\infty}(0,t;L^{2q}_{\Omega})}^2
        \overset{\eqref{eq:a1est}}&{\lesssim} \lVert v_{1,0} \rVert_{L^{2q}_{\Omega}}^2 + q \beta^2 t,
    \intertext{and}
        \Bigl\lVert \int_0^t P(t-t') \Pi \beta^2 u^*\d t' \Bigr\rVert_{L^q_{\Omega}(L^2_x)} \overset{\eqref{eq:Pdecay}}&{\lesssim} \beta^2 t.
    \end{align*}
    It only remains to estimate the stochastic integrals in \eqref{eq:w2decomp}.
    For the first we have
    \begin{align*}
        \Bigl\lVert \int_0^t  P(t-t')\Pi i w_1 \Phi \d W(t') \Bigr\rVert_{L^q_{\Omega}(L^2_x)} 
        \overset{\eqref{eq:stochdecayconvest}}&{\lesssim} \sqrt{q}\beta t^{\frac{1}{2}}\lVert w_1 \rVert_{L^{\infty}(0,t;L^q_{\Omega}(L^2_x))} \\
        \leq \tfrac{1}{2} \lVert w_1 \rVert_{L^{\infty}(0,t;L^q_{\Omega}(L^2_x))}^2 + \tfrac{1}{2} q \beta^2 t
        \overset{\eqref{eq:w1est}}&{\lesssim} \lVert v_{1,0} \rVert_{L^{q}_{\Omega}(L^2_x)}^2 + q \beta^2 t,
    \intertext{and for the second}
        \Bigl\lVert \int_0^t  P(t-t')\Pi i a_1 u^*_x \Phi \d W(t') \Bigr\rVert_{L^q_{\Omega}(L^2_x)} 
        \overset{\eqref{eq:stochdecayconvest}}&{\lesssim} \sqrt{q} \beta t^{\frac{1}{2}} \lVert a_1 \rVert_{L^{\infty}(0,t;L^q_{\Omega})} \\
        \leq \tfrac{1}{2} \lVert a_1 \rVert_{L^{\infty}(0,t;L^q_{\Omega})}^2 + \tfrac{1}{2} q \beta^2 t 
        \overset{\eqref{eq:a1est}}&{\lesssim} \lVert v_{1,0} \rVert_{L^q_{\Omega}(L^2_x)}^2 + q \beta^2 t. \qedhere
    \end{align*}
\end{proof}
\begin{proof}[Proof of Proposition~\ref{prop:longterm}]
From our previous ansatz for $u$ and $v_1$ we have the equalities
\begin{subequations}
\begin{align}
\label{eq:uustar1}
    u(t) - u^*(x+\sigma a_1(t)) \overset{\eqref{eq:defzprime}}&{=} u^* - u^*(x+\sigma a_1(t)) + \sigma v_1(t) + z'(t) \\
    \label{eq:uustar2}
    \overset{\eqref{eq:ansatzv1}}&{=} u^* + \sigma a_1(t)u^*_x - u^*(x+\sigma a_1(t)) + \sigma w_1(t) + z'(t).
\end{align}
\end{subequations}
From \eqref{eq:uustar1} and a zeroth-order Taylor expansion we may obtain
\begin{subequations}
\begin{align}
    \label{eq:uustarest1}
    \lVert u(t) - u^*(x+\sigma a_1(t)) \rVert_{L^2_x} \leq C_1 \sigma \lvert a_1(t) \rvert + \sigma \lVert v_1(t) \rVert_{L^2_x} + \lVert z'(t) \rVert_{L^2_x},
\end{align}
for some constant $C_1$ derived from $u^*$.
From \eqref{eq:uustar2} and a first-order Taylor expansion we also get
\begin{align}
    \label{eq:uustarest2}
    \lVert u(t) - u^*(x+\sigma a_1(t)) \rVert_{L^2_x} \leq C_2 \sigma^2 \lvert a_1(t) \rvert^2 + \sigma \lVert w_1(t) \rVert_{L^2_x} + \lVert z'(t) \rVert_{L^2_x},
\end{align}
\end{subequations}
for some constant $C_2$ also derived from $u^*$.
Now set $T = a^{-1}\log(6M)$, where $a$ and $M$ are the constants from \eqref{eq:Pdecay}, and fix some $c_1$, $c_2$, $\eps'$ such that Theorem~\ref{thm:asym2} holds with this choice of $T$ (note that our initial condition corresponds to setting $v_{1,0} = \sigma^{-1}v_0$).
Additionally, set $\tilde{c}_1 = \tfrac{1}{6}\min\{M^{-1}, C_1^{-1} \lVert \mathcal{P} \rVert_{\mathcal{L}(L^2_x;\R)} \}$.
From the assumption that $\lVert v_0 \rVert \leq \tilde{c}_1 \eps$ we obtain
\begin{align*}
    \sigma \lvert a_1(t) \rvert \overset{\eqref{eq:a1def}}&{\leq} C_1^{-1} \frac{\eps}{6}  + \sigma  \lVert \mathcal{P} \rVert_{\mathcal{L}(L^2_x;\R)} \, \Bigl \lVert \int_0^t u^* \Phi \d W(t')\Bigr\rVert_{L^2_x}, \\
    \sigma \lVert v_1(t) \rVert_{L^2_x} \overset{\eqref{eq:v1}}&{\leq} \frac{\eps}{6} + \sigma \Bigl\lVert \int_0^t P(t-t')u^* \Phi \d W(t') \Bigr\rVert_{L^2_x}, \\
    \sigma \lVert w_1(T) \rVert_{L^2_x} \overset{\eqref{eq:w1def}}&{\leq} \tilde{c}_1 \frac{\eps}{6} + \sigma \Bigl\lVert \int_0^T P(T-t')\Pi u^* \Phi \d W(t') \Bigr\rVert_{L^2_x},
\end{align*}
where the third inequality follows from \eqref{eq:Pdecay} since $Me^{-aT} = \frac{1}{6}$ by our choice of $T$.
Using \eqref{eq:PPstrichstoch}, \eqref{eq:PP0strichstoch}, and Lemma~\ref{lem:tail}, we can find constants $\lambda$, $c_2' > 0$, such that
\begin{subequations}
\label{eq:everythingsmall}
\begin{align}
    \label{eq:a1small}
    \mathbb{P}\bigl[C_1 \sigma \lvert a_1 \rvert_{L^{\infty}(0,T)} \geq \tfrac{\eps}{3}\bigr] \leq \exp(-c_2' \sigma^{-2}\eps^2), \\
    \label{eq:v1small}
    \mathbb{P}\bigl[\sigma \lVert v_1 \rVert_{L^{\infty}(0,T;L^2_x)} \geq \tfrac{\eps}{3}\bigr] \leq \exp(-c_2' \sigma^{-2}\eps^2), \\
    \label{eq:w1small}
    \mathbb{P}\bigl[\sigma \lVert w_1(T) \rVert_{L^2_x} \geq \tilde{c}_1 \tfrac{\eps}{3} \bigr] \leq \exp(-c_2' \sigma^{-2}\eps^2),
\end{align}
whenever $\sigma^{-1}\eps \geq \lambda$.
If we take $\eps'$ small enough such that $\tilde{c}_1 \tfrac{\eps'}{3} \geq c_1 \eps'^2$ (if necessary), then by Theorem~\ref{thm:asym2}, this also results in
\begin{equation}
\begin{aligned}
    \label{eq:zprimesmall}
    \mathbb{P}\bigl[\lVert z'\rVert_{L^{\infty}(0,T;L^2_x)} \geq \tilde{c}_1\tfrac{\eps}{3}\bigr] 
    &\leq \mathbb{P}\bigl[\lVert z'\rVert_{L^{\infty}(0,T;L^2_x)} \geq c_1 \eps^2\bigr] \\
    &= \mathbb{P}\bigl[\tau_{z'} < T \bigr] \\
    &\leq \mathbb{P}\bigl[\tau_{z'} < \tau_{v_1}\bigr] + \mathbb{P}\bigl[\tau_{v_1} < T \bigr] \\
    \overset{\eqref{eq:zprimeprobest},\eqref{eq:v1small}}&{\leq} \exp(-c_2 \sigma^{-2}\eps^2) + \exp(-c_2' \sigma^{-2}\eps^2),
\end{aligned}
\end{equation}
for all $\eps \leq \eps'$.
If we additionally take $\eps'$ smaller (if necessary) such that $\frac{C_1\sqrt{3\tilde{c}_1}}{\sqrt{C_2\eps'}} \geq 1$, then we also get
\begin{equation}
    \label{eq:a1small2}
      \mathbb{P}\bigl[C_2 \sigma^2\lvert a_1 \rvert_{L^{\infty}(0,T)}^2 \geq \tilde{c}_1\tfrac{\eps}{3}\bigr] 
      = \mathbb{P}\Bigl[C_1 \sigma\lvert a_1 \rvert_{L^{\infty}(0,T)} \geq \tfrac{C_1 \sqrt{3\tilde{c}_1}}{\sqrt{C_2\eps}} \tfrac{\eps}{3}\Bigr]
      \overset{\eqref{eq:a1small}}{\leq} \exp(-c_2' \sigma^{-2}\eps^2),
\end{equation}
\end{subequations}
for all $\eps \leq \eps'$.
Equation \eqref{eq:uustarest1}, a simple union bound and the fact that $\tilde{c}_1 \leq 1$ now gives
\begin{align*}
    \mathbb{P}[\lVert u(\cdot) - u^*(x + \sigma a_1(\cdot)) \rVert_{L^{\infty}(0,T;L^2_x)} \geq \eps]
    &\leq \mathbb{P}\bigl[C_1 \sigma\lvert a_1 \rvert_{L^{\infty}(0,T)} \geq \tfrac{\eps}{3}\bigr] \\
    &\quad+ \mathbb{P}\bigl[\sigma \lVert v_1 \rVert_{L^{\infty}(0,T;L^2_x)} \geq \tfrac{\eps}{3}\bigr] \\
    &\quad+ \mathbb{P}\bigl[\lVert z' \rVert_{L^{\infty}(0,T;L^2_x)} \geq \tilde{c}_1 \tfrac{\eps}{3}\bigr]\\
    \overset{\eqref{eq:everythingsmall}}&{\leq} 3 \exp(-c_2' \sigma^{-2} \eps^2) + \exp(-c_2 \sigma^{-2}\eps^2).
\intertext{Similarly, from \eqref{eq:uustarest2} we get}
    \mathbb{P}[\lVert u(T) - u^*(x+\sigma a_1(T)\rVert_{L^2_x} \geq \tilde{c}_1 \eps \bigr]
    &\leq \mathbb{P}\bigl[C_2 \sigma^2\lvert a_1(T) \rvert^2 \geq \tilde{c}_1 \tfrac{\eps}{3}\bigr]\\
    &\quad+ \mathbb{P}\bigl[\sigma \lVert w_1(t) \rVert_{L^2_x} \geq \tilde{c}_1 \tfrac{\eps}{3}\bigr] \\
    &\quad+ \mathbb{P}\bigl[\lVert z'(T) \rVert_{L^2_x} \geq \tilde{c}_1 \tfrac{\eps}{3}\bigr]\\
    \overset{\eqref{eq:everythingsmall}}&{\leq} 3 \exp(-c_2' \sigma^{-2} \eps^2) + \exp(-c_2 \sigma^{-2}\eps^2).
\end{align*}
(note that although we wrote $L^{\infty}(0,T)$ in \eqref{eq:everythingsmall}, we could have also written $C([0,T])$ so the estimate is valid).
The result follows by choosing $\tilde{c}_2 = \min\{c_2, c_2'\}$.
\end{proof}

\appendix
\section{Hilbert--Schmidt operators}
\label{app:hs}
\begin{proof}[Proof of Proposition~\ref{prop:hs}]
Fix some $\phi \in L^2(\R;\R)$, and define for any $\psi \in L^2_x$ the following map:
\begin{align*}
    \Phi_{\psi} \colon f \mapsto \psi * f.
\end{align*}
Recall that with this notation $\Phi = \Phi_{\phi}$ (see \eqref{eq:defPhi}).
Now let $e_k$, $k \in \N$ be any orthonormal basis of $L^2(\R;\R)$. 
We see using Parseval's identity that
\begin{equation*}
        \sum_{k\in \mathbb{N}}(\Phi e_k(x))^2 = \sum_{k\in \mathbb{N}}\langle \phi(\cdot - x), e_k \rangle^2_{L^2_x} 
        = \lVert \phi(\cdot - x )\rVert_{L^2_x}^2 \overset{\eqref{eq:defbeta}}{=} \beta^2,
\end{equation*}
which shows \eqref{eq:Fbeta}.
Using Fubini's theorem and Parseval's identity, we can also compute
\begin{align*}
        \lVert u\Phi \rVert_{\mathcal{L}_2(L^2(\R;\R);L^2_x)}^2
        &=\sum_{k \in \N} \lVert u \Phi e_k \rVert_{L^2_x}^2 
        =\sum_{k \in \N} \int_{\R}\lvert u(x) \rvert^2 \langle \phi(\cdot - x), e_k \rangle_{L^2_x}^2\d x \\
        &= \int_{\R}\lvert u(x) \rvert^2 \sum_{k \in \N} \langle \phi(\cdot - x), e_k \rangle_{L^2_x}^2\d x 
        = \int_{\R}\lvert u(x) \rvert^2 \lVert \phi(\cdot - x) \rVert_{L^2_x}^2 \d x \\
        &= \lVert u \rVert_{L^2_x}^2 \lVert \phi \rVert_{L^2_x}^2,
\end{align*}
which shows \eqref{eq:hilbertschmidteq}.

To show \eqref{eq:hilbertschmidtest} we will make use of complex interpolation.
Thus, we will now break convention and regard $H^s_x$ and $L^2_x$ as complex spaces for the rest of this section.
We will show the complexified estimate
\begin{equation}
    \label{eq:hilbertschmidtcomplex}
    \lVert u \Phi \rVert_{\mathcal{L}_2(L^2_x;H^s_x)} \leq C_s \lVert \phi \rVert_{H^s_x} \lVert u \rVert_{H^s_x}.
\end{equation}
The result then follows after noting that an orthonormal basis of the real Hilbert space $L^2(\R;\R)$ is also an orthonormal basis of $L^2_x$ when the latter is regarded as a complex Hilbert space.
We first show by induction that \eqref{eq:hilbertschmidtcomplex} holds when $s = 2n$ for some nonnegative integer $n$.
By repeating the previous calculation, we find again that
\begin{align*}
    \lVert u \Phi \rVert_{\mathcal{L}_2(L^2_x;L^2_x)} = \lVert u \rVert_{L^2_x}\lVert \phi \rVert_{L^2_x},
\end{align*}
which implies the base case.
Therefore, we now assume that the statement holds for some $n$.
By elementary computations, we find
\begin{align*}
        (1-\Delta) (u \Phi f) 
        &= (1-\Delta)(u (\phi * f)) \\
        &= u (\phi * f) - \Delta u (\phi * f) - 2\, \partial_x u (\partial_x \phi * f) - u (\Delta \phi * f) \\
        &= u \Phi f - \Delta u \Phi f - 2\, \partial_x u (\Phi_{\partial_x\phi}f) - u (\Phi_{\Delta \phi}f),
\intertext{so that}
    (1-\Delta) (u \Phi) &= u \Phi - \Delta u \Phi - 2\, \partial_x u \Phi_{\partial_x\phi} - u \Phi_{\Delta\phi}.
\end{align*}
Combining this with the triangle inequality and the induction hypothesis gives
\begin{align*}
        \lVert u \Phi \rVert_{\mathcal{L}_2(L^2_x;H^{n+2}_x)}
        &= \lVert (1-\Delta)(u \Phi)\rVert_{\mathcal{L}_2(L^2_x;H^n_x)} \\
        &\leq C \bigl(\lVert u \rVert_{H^n_x} \lVert \phi \rVert_{H^n_x}
        + \lVert \Delta u \rVert_{H^n_x} \lVert \phi \rVert_{H^n_x}
        + 2\lVert \partial_x u \rVert_{H^n_x} \lVert \partial_x \phi \rVert_{H^n_x}
        + \lVert u \rVert_{H^n_x} \lVert \Delta \phi \rVert_{H^n_x} \bigr) \\
        &\leq C' \lVert u \rVert_{H^{n+2}_x}\lVert \phi \rVert_{H^{n+2}_x}.
\end{align*}
Now let $s \in [0,\infty)$ be arbitrary, let $n$ be an integer such that $2n \geq s$, let $\theta \in [0,1]$ be such that $s = 2n\theta$, and consider the bilinear map
\begin{equation*}
        B \colon (u,\phi) \mapsto u \cdot \Phi_{\phi}.
\end{equation*}
We have already shown that $B$ is bounded from $L^2_x \times L^2_x$ to $\mathcal{L}_2(L^2_x;L^2_x)$ and from ${H^{2n}_x \times H^{2n}_x}$ to ${\mathcal{L}_2(L^2_x,H^{2n}_x)}$.
Thus, by complex interpolation (using the notation $[\cdot,\cdot]_{\theta}$ for the intermediate space) it follows that $B$ is also bounded from
\begin{align*}
[L^2_x,H^{2n}_x]_{\theta} \times [L^2_x,H^{2n}_x]_{\theta} =  H^{s}_x \times H^{s}_x \\
\end{align*}
to
\begin{align}
\label{eq:hsinterpol}
[\mathcal{L}_2(L^2_x,L^{2}_x), \mathcal{L}_2(L^2_x,H^{2n}_x)]_{\theta} = \mathcal{L}_2(L^2_x,H^{s}_x).
\end{align}
For the interpolation of bilinear operators we have used~\cite[Theorem 4.4.1]{bergh_interpolation_1976}, and the isomorphism \eqref{eq:hsinterpol} is shown for $\gamma$-radonifying operators (which generalize Hilbert--Schmidt operators) in~\cite[Theorem 9.1.25]{hytonen_analysis_2016}.
\end{proof}

\section{Stochastic Strichartz estimates}
\label{app:stochstrich}
To prove \eqref{eq:strichstoch} we distinguish between the cases $p = 2$ and $p > 2$.
\begin{proof}[Case $p > 2$]
For every $t' \in [0,T]$, define the operator
\begin{align*}
        \Psi(t') : H^{s}_x &\to L^r(0,T;H^{s,p}_x) \\
        \psi &\mapsto 1_{[t',T]}(\cdot)S(\cdot - t')\psi,
\end{align*}
and observe that $\lVert \Psi(t') \rVert_{\mathcal{L}(H^{s}_x;L^r(0,T;H^{s,p}_x))} 
\leq \lVert \Psi(0) \rVert_{\mathcal{L}(H^{s}_x;L^r(0,T;H^{s,p}_x))} \leq L$
for some $L < \infty$ which is independent of $T$ by \eqref{eq:strichhom}.

Since $p \in (2,\infty)$, the space $L^p_x$ is $2$-smooth~\cite[Proposition 3.5.30]{hytonen_analysis_2016}.
Using the lifting operator $(1 - \Delta)^{\frac{s}{2}}$, this property immediately extends to $H^{s,p}_x$.
Since $r \in (4,\infty)$, the space $L^r(0,T;H^{s,p}_x)$ has this property as well (see for instance~\cite[Proposition 2.2]{van_neerven_maximal_2022}).
Thus, using our definition of $\Psi$ we can rewrite and estimate
\begin{align*}
    \Bigl\lVert \int_0^{\cdot} S(\cdot-t') h(t')\Phi & \d W(t') \Bigr\rVert_{L^q_{\Omega}(L^r(0,T;H^{s,p}_x))} 
    = \Bigl\lVert \int_0^T \Psi(t')h(t')\Phi \d W(t') \Bigr\rVert_{L^q_{\Omega}(L^r(0,T;H^{s,p}_x))} \\
    &\leq C \sqrt{q} \lVert \Psi h \Phi \rVert_{L^q_{\Omega}(L^2(0,T;\gamma(L^2(\R;\R);L^r(0,T;H^{s,p}_x))))} \\
    &\leq C L \sqrt{q} \lVert h \Phi \rVert_{L^q_{\Omega}(L^2(0,T;\mathcal{L}_2(L^2(\R;\R);H^{s}_x)))} \\
    \overset{\eqref{eq:hilbertschmidtest}}&{\leq} C' L \sqrt{q} \lVert \phi \rVert_{H^{s}_x} \lVert h \rVert_{L^q_{\Omega}(L^2(0,T;H^{s}_x))}.
\end{align*}
The first inequality follows from~\cite[Theorem 1.1]{seidler_exponential_2010}, and the second follows from the left-ideal property of $\gamma$-radonifying operators (which can easily be seen from the definition) and the boundedness of $\Psi$.
\end{proof}
\begin{proof}[Case $p = 2$]
Since $(r,p)$ satisfies \eqref{eq:defadmissible} we have $r = \infty$.
Using the fact that $S(t)$ is unitary on $H^{s}_x$ and using~\cite[Theorem 1.1]{seidler_exponential_2010} again we find
\begin{align*}
    \Bigl\lVert \int_0^{\cdot} S(\cdot-t')h(t')\Phi & \d W(t') \Bigr\rVert_{L^q_{\Omega}(L^{\infty}(0,T;H^{s}_x))}
    = \Bigl\lVert \int_0^{\cdot} S(-t')h(t')\Phi \d W(t') \Bigr\rVert_{L^q_{\Omega}(L^{\infty}(0,T;H^{s}_x))} \\
    &\leq C \sqrt{q} \lVert S(-\cdot)h(\cdot)\Phi \rVert_{L^q_{\Omega}(L^2(0,T;\mathcal{L}_2(L^2(\R;\R);H^{s}_x)))} \\
    &= C \sqrt{q} \lVert h\Phi \rVert_{L^q_{\Omega}(L^2(0,T;\mathcal{L}_2(L^2(\R;\R);H^{s}_x)))} \\
    \overset{\eqref{eq:hilbertschmidtest}}&{\leq} C' \sqrt{q} \lVert \phi \rVert_{H^{s}_x} \lVert h \rVert_{L^q_{\Omega}(L^2(0,T;H^{s}_x))}.
\end{align*}
The continuity in $H^{s}_x$ follows by a routine approximation argument.
\end{proof}

\printbibliography

\end{document}